\documentclass{amsart}
\allowdisplaybreaks

\usepackage{amsmath,amssymb}
\usepackage{multirow}
\usepackage{graphicx}

\usepackage{color} %%%{{{ color and hyperref
\definecolor{pdflinkcolor}{rgb}{.1,.6,.1}	% darkgreen
\definecolor{pdfcitecolor}{rgb}{.6,.1,.1}	% darkred
\definecolor{pdfanchorcolor}{rgb}{0,1,0}	% green
\definecolor{pdfurlcolor}{rgb}{.1,.6,.1}	% darkgreen
\definecolor{pdfpagecolor}{rgb}{0,0,1}		% blue
\definecolor{pdffilecolor}{rgb}{1,0,0}		% red
\usepackage[%
   breaklinks,
   colorlinks=true, 
   citecolor=pdfcitecolor,
   linkcolor=pdflinkcolor,
   anchorcolor=pdfanchorcolor,
   urlcolor=pdfurlcolor,
   filecolor=pdffilecolor,
   pagebackref=false,
   pdfauthor={},
   pdftitle={Sharp eigenvalue enclosures for the perturbed angular Kerr-Newman Dirac operator},]{hyperref}
\usepackage{colortbl}
\definecolor{mygrey}{gray}{0.8}

\newcommand{\mc}[1]{\mathcal{#1}}

\newtheorem*{theorem*}{Theorem}
\newtheorem{lem}{Lemma}
\newtheorem{thm}[lem]{Theorem}

\newtheorem{cor}[lem]{Corollary}
\newtheorem{rem}[lem]{Remark}

\renewcommand{\Re}{\operatorname{Re}}
\renewcommand{\Im}{\operatorname{Im}}

\newcommand{\rd}{\mathrm d}
\newcommand{\e}{\mathrm e}
\newcommand{\id}{I_{2\times 2}}
\newcommand{\I}{\mathrm i}
\newcommand{\C}{\mathbb C}
\newcommand{\N}{\mathbb N}
\newcommand{\R}{\mathbb R}
\newcommand{\Z}{\mathbb Z}

\newcommand{\Dop}{\mathcal D}
\newcommand{\Sopk}{\mathcal S_\kappa}

\newcommand{\Bk}{B_{\kappa}}

\newcommand{\mAk}{\mathcal A_{\kappa}}
\newcommand{\mfAk}{\mathfrak A_{\kappa}}
\newcommand{\mfRk}{\mathfrak R_{\kappa}}
\newcommand{\mfA}{\mathcal A}
\newcommand{\mfR}{\mathcal R}
\newcommand{\pot}{V}
\newcommand{\triple}[1]{%
{\left\vert\kern-0.25ex\left\vert\kern-0.25ex\left\vert #1 \right\vert\kern-0.25ex\right\vert\kern-0.25ex\right\vert}}

\newcommand{\Ltwo}{L_2(0,\pi)}
\newcommand{\Ltwotwo}{\Ltwo^2}

\DeclareMathOperator{\integral}{J}
\DeclareMathOperator{\dom}{dom}
\DeclareMathOperator{\spec}{spec}
\DeclareMathOperator{\sign}{sign}

\date{3rd August 2015}

\title[Sharp eigenvalue enclosures for the K-N operator]{Sharp eigenvalue enclosures for the perturbed angular Kerr-Newman Dirac operator}

\author{Lyonell Boulton}
\author{Monika Winklmeier}

\begin{document}

\maketitle
\begin{abstract}
   We examine a certified strategy for determining sharp intervals of enclosure for the eigenvalues of matrix differential operators with singular coefficients.
   The strategy relies on computing the second order spectrum relative to subspaces of continuous piecewise linear functions.
   For smooth perturbations of the angular Kerr-Newman Dirac operator, explicit rates of convergence linked to regularity of the eigenfunctions are established.
   Numerical tests which validate and sharpen by several orders of magnitude the existing benchmarks are also included.
\end{abstract}

% \tableofcontents

\section{Introduction}
%%%{{{
The Kerr-Newman spacetime describes a stationary electrically charged rotating black hole.
In this regime the Dirac equation for an electron takes the form
\begin{align*}
   (\widehat\mfA + \widehat\mfR)\widehat\Phi = 0
\end{align*}
where $\widehat\Phi$ is a four components spinor which describes the wave function of the electron. The operators $\widehat\mfA$ and $\widehat\mfR$ are complicated $4\times 4$ differential expressions in $(r_+,\infty)\times(0,\pi)\times(-\pi,\pi)\times\R$, \cite{chandrasekhar}.
From the ansatz \[\widehat \Phi(r,\theta,\phi,t) = \e^{-\I \omega t} \e^{-\I \kappa\phi}\Phi(r,\theta)\] with a suitable four components $\Phi$, two eigenvalue equations are obtained:
\begin{align*}
   (\mathcal \mfRk -\omega)\rho = 0 \qquad \text{and} 
   \qquad
   (\mathcal \mfAk -\lambda)\psi = 0.
\end{align*}
The radial part $\mfRk$ comprises only derivatives with respect to the radial coordinate and the angular part $\mfAk$ has only derivatives with respect to the angular coordinate $\theta$. Note that $\theta=0$ and $\theta=\pi$ indicate the direction parallel to the axis of rotation of the black hole.
The eigenvalue $\omega$ in the radial equation corresponds to the energy of the electron.
These two equations are not completely separated as they are still coupled by the angular momentum of the rotating black hole which is given by a real parameter $a$.

The Cauchy problem associated to the full Kerr-Newman Dirac operator has been considered in
\cite{FKSY2000,FKSYErr2000,BaticSchmid2006,WY09}, while the radial part of the 
system has been thoroughly examined in \cite{Schmid,WinklmeierPhD}. In the present paper, 
we focus on the eigenvalue problem associated to the angular part which in suitable coordinates can be written as
\begin{align} \label{frambuesa}
   \mfAk = 
     \begin{pmatrix}
      -am\cos\theta & \frac{\rd}{\rd \theta} + \frac{ \kappa}{ \sin\theta } + a\omega\sin\theta \\
      -\frac{\rd}{\rd \theta} + \frac{ \kappa}{ \sin\theta } + a\omega \sin\theta & am\cos\theta
   \end{pmatrix}, \qquad 0<\theta<\pi.
\end{align}
The only datum inherent to the black hole in this expression is the coupling parameter $a$.
The other physical quantities are the mass of the electron $m$, its energy $\omega$ and $\kappa\in\Z+\frac{1}{2}$.
%which describes its angular momentum around the axis of symmetry of the system.
% Here $\omega$ and $\kappa$ arise from the  separation of variables.

The operator which will be associated to \eqref{frambuesa} below is self-adjoint, it has a compact resolvent and it is strongly indefinite 
in the sense that the spectrum accumulates at $\pm\infty$.
Various attempts at computing its spectrum have been considered in the past.
Notably, a series expansion for $\lambda$ in terms of $a(m+\omega)$ and $a(m-\omega)$ 
was derived in \cite{SFC} by means of techniques involving continued fractions, see also \cite{BSW2005}. 
A further asymptotic expansion in terms of $a\omega$ and $m/\omega$ was
 reported in \cite{chakrabarti}.
In both cases however, no precise indication of the orders of magnitude of a reminder term was given.

A simple explicit expression for the eigenvalues appears to be available only for the case $am=\pm a\omega$.
By invoking an abstract variational principle on the corresponding operator pencil, coarse analytic enclosures for these eigenvalues in the case $am\neq\pm a\omega$ were found in \cite{WinklmeierPhD,Win08}.
Our aim is to sharpen these enclosures by several orders of magnitude via a projection method. 

Techniques for determining enclosures for eigenvalues of indefinite operator matrices via variational formulations have been examined by many authors in the past, see for example 
\cite{Siedentop,EstebanSere,LLT2002,KLT2004,LangerMTretter2006,Tretterbook,MR2995208}.
These are strongly linked with the classical complementary bounds for eigenvalues by Temple and Lehmann \cite[Theorem~4.6.3]{1995Davies}, which played a prominent role in the early days of quantum mechanics.
See \cite{1995Zimmermann,2004Davies}.
The so-called quadratic method, developed by Davies \cite{1998Davies}, Shargorodsky \cite{2000Shargorodsky} and others \cite{2004Levitin,MR2219033}, is an alternative to these approaches.
As we shall demonstrate below, an application of this method leads to sharp eigenvalue bounds for the operator associated to $\mfAk$.
Recently, the quadratic method was applied successfully to crystalline Schr{\"o}dinger operators \cite{2006Boulton}, the hydrogenic Dirac operator \cite{Boulton:2009p2970} and models from magnetohydrodynamics \cite{2010Strauss}.

The concrete purpose of this paper is to address the numerical calculation of intervals of enclosure for the eigenvalues of $\mfAk$ with the possible addition of a smooth perturbation.
We formulate an approach which is certified up to machine precision and is fairly general in character.
We also find explicit rates for its convergence in terms of the regularity of the eigenfunctions. In the case of the unperturbed $\mfAk$, we perform various numerical tests which validate and sharpen existing benchmarks by several orders of magnitude. 

In the next section we present the operator theoretical setting of the eigenvalue problem.
Lemma~\ref{lem:decay-eigenfunctions} and Corollary~\ref{cor:H0} are devoted to explicit  smoothness properties and boundary behaviour of the eigenfunctions.
We include a complete proof of the first statement in the appendix~\ref{app:frobenius}.

In Section~\ref{section:manzana} we formulate the quadratic method on trial subspaces of piecewise linear functions. Theorem~\ref{mamoncillo} establishes concrete rates of convergence for the numerical approximation of eigenvalues.
A proof of this crucial statement is deferred to Section~\ref{section:ceresa}.
The main ingredients of this proof are the explicit error estimates for the approximation of eigenfunctions by continuous piecewise linear functions in the graph norm which are presented in Theorem~\ref{thm:main}.

Various numerical tests can be found in Section~\ref{platano}.
We begin that section by describing details of the calculations reported previously in \cite{chakrabarti,SFC}.
These tests address the following.
\begin{enumerate}
    \item Validity of the numerical values from \cite{chakrabarti,SFC}.
    \item Sharpening of the eigenvalue bounds in the context of the quadratic method.
    \item Optimal order of convergence.
\end{enumerate}
These tests were performed by implementing in a suitable manner the  computer code written in Comsol LiveLink which is included in the Appendix~\ref{mandarina}.

\subsection*{Notational conventions and basic definitions}

Below we employ calligraphic letters to refer to operator matrices.
We denote by $\dom(A)$ the domain of the linear operator $A$.
The Hilbert space $\Ltwotwo$ is that consisting of two-component vector-valued functions 
$u:(0,\pi)\longrightarrow \mathbb{C}^2$ such that 
\[
   \|u\| =\left( \int_0^\pi | u(\theta) |^2\rd \theta \right)^{\frac{1}{2}}
= \left( \int_0^\pi | u_1(\theta) |^2+|u_2(\theta)|^2\,\rd \theta \right)^{\frac{1}{2}}<\infty.
\]
Let $u\in \Ltwotwo$ and denote its Fourier coefficients in the sine basis by
\begin{align*}
   \widehat u_n = \sqrt{\frac{2}{\pi}} \int_0^\pi u(\theta) \sin( n \theta)\,\rd \theta\in \mathbb{C}^2,
   \qquad n\in\N.
\end{align*}
Let $\langle n \rangle = (1+n^2)^{\frac{1}{2}}$.
Let $r>0$.
The fractional  Sobolev spaces $H^r(0,\pi)$ will be, by definition, the Hilbert space
\begin{equation*}
   H^r(0,\pi)
   =
   \left\{ u\in \Ltwotwo : \sum_{n\in\N} \langle n\rangle^{2r} |\widehat u_n|^2 < \infty \right\}.
\end{equation*}
Here the norm is given by the expression
\[
    \|u\|_r=\left(\sum_{n\in\N} \langle n\rangle^{2r} |\widehat u_n|^2\right)^{\frac12}.
\]
Note that an analogous definition can be made, if we instead consider the Fourier coefficients of $u$ in the cosine basis.

If $r\in \N$, we recover the classical Sobolev spaces, where the norm has also the representation
\begin{align*}
   \|u\|_r = \left( \sum_{j=0}^r \|u^{(j)}\|^2 \right)^{\frac{1}{2}}.
\end{align*}
We set $H^1_0(0,\pi)$ to be the completion of $[C^\infty_0(0,\pi)]^2$ in the norm of $H^1(0,\pi)$.

%%%}}}

\section{A concrete self-adjoint realisation and regularity of the eigenfunctions}
%%%{{{

Here and everywhere below $\kappa$ will be a real parameter satisfying $|\kappa|\geq \frac12$ and $\pot= [v_{ij}]_{i,j=1}^2$ will be a hermitian matrix potential with all its entries being complex analytic functions in a  neighbourhood of $[0,\pi]$.
The operator theoretical framework of the spectral problem associated to matrices of the form
\begin{align}
   \label{grosella}
   \mfAk =
   \begin{pmatrix}
      0 & \frac{\rd}{\rd \theta} + \frac{ \pi\kappa}{ \theta(\pi-\theta) } \\
      -\frac{\rd}{\rd \theta} + \frac{ \pi\kappa }{\theta(\pi-\theta) } & 0
   \end{pmatrix}
    + \pot 
\end{align}
can be set by means of well establish techniques, \cite{Weidmann}.
Our first goal is to identify a concrete self-adjoint realisation of the differential expression \eqref{grosella} in $\Ltwotwo$.

\begin{rem}
The spectral problem associated to the angular Kerr-Newman Dirac operator \eqref{frambuesa} fits into the present framework by taking  
\begin{equation} \label{coco}
      V(\theta)=\kappa \left( \frac{1}{\sin(\theta)}-\frac{\pi}{\theta(\pi-\theta)}  \right)\begin{pmatrix}
              0 & 1 \\ 1 & 0 
          \end{pmatrix} + \begin{pmatrix}
          -am \cos(\theta)    & a\omega \sin(\theta) \\ a\omega \sin(\theta) & am \cos(\theta) 
         \end{pmatrix}.
\end{equation}
% in the physically relevant regime $\kappa\in \mathbb{Z}+\frac12$. 
Note that $V$ has an analytic continuation to a heighbourhood of\, $[0, \pi]$.
\end{rem}

Let $\pot=0$. In this case the fundamental solutions of $\mfAk \Psi =0$ can be found explicitly.
The differential expression $\mfAk$ is in the limit point case for $|\kappa|\geq \frac12$ and in the limit circle case for $|\kappa|< \frac12$. 
Thus, for $|\kappa|\ge \frac{1}{2}$, the maximal operator
\begin{equation} \label{agraz}
   \mAk= \mfAk|_{\dom(\mAk)},\quad
   \dom(\mAk) =
   \left\{ \Psi\in \Ltwotwo \ :\ \mfAk\Psi\in \Ltwotwo \right \}
\end{equation}
is self-adjoint in $\Ltwotwo$.

By virtue of the particular block operator structure of the matrix in \eqref{grosella}, 
$
 \dom(\mAk)=\mathcal{D}_1 \oplus\mathcal{D}_2,
$
where
\begin{align*}
   \mathcal{D}_1   &=
   \left \{ f \in \Ltwo\ :\
   \int_0^\pi \left|\left( -\frac{\rd}{\rd \theta} + \frac{ \pi \kappa }{ \theta(\pi-\theta) } \right)f(\theta)\right|^2 \rd \theta  < \infty \right\} \quad \text{and} \\
   \mathcal{D}_2  
   &= \left\{ 
   f \in \Ltwo\ :\
\int_0^\pi \left|\left( \frac{\rd}{\rd \theta} + \frac{ \pi \kappa }{ \theta(\pi-\theta) } \right)f(\theta)\right|^2 \rd \theta  < \infty
\right\}. 
\end{align*}
Thus, the operators
\begin{alignat}{2}
   \label{mangostino}
   \Bk & = \frac{\rd}{\rd \theta} + \frac{\pi \kappa }{ \theta(\pi-\theta) },
   &\qquad \dom(\Bk)=\mathcal{D}_2, \\
   \label{mangostinodagger}
   \Bk^\dagger & = -\frac{\rd}{\rd \theta} + \frac{ \pi \kappa }{ \theta(\pi-\theta) },
   & \dom(\Bk^\dagger )=\mathcal{D}_1 ,
\end{alignat}
are adjoint to one another and
\[
    \mAk=\begin{pmatrix} 0 & \Bk \\ \Bk^\ast & 0 \end{pmatrix}.
\]
Both $\Bk$ and $\Bk^*$ have empty spectrum. The resolvent kernel of these expressions is square integrable, so they have compact resolvent.
% See \eqref{pera}-\eqref{perast} in Appendix~\ref{app:properties}.
Therefore also $\mAk$ has a compact resolvent.

Now consider $\pot \not=0$. We define the corresponding operator associated with
\eqref{grosella} also by means of \eqref{agraz}.
As $V$ is bounded, it yields a bounded self-adjoint matrix multiplication operator in $\Ltwotwo$.
Routine perturbation arguments show that also in this case $\mAk$ is a self-adjoint operator with compact resolvent.
Note that $\dom(\mAk)$ is independent of $V$.

\begin{rem}
   \label{naranja}
 The spectrum of $\mAk$ consists of two sequences of eigenvalues. One non-negative, accumulating at $+\infty$, and the other one negative accumulating at $-\infty$. An explicit analysis involving the Frobenius method (see Remark~\ref{grapefruit}) shows that no eigenvalue of $\mAk$ has multiplicity greater than one.
\end{rem}

As we shall see next, any eigenfunction of $\mAk$ is regular in the interior of $[0,\pi]$.
Moreover, it has a boundary behaviour explicitly controlled  by $|\kappa|$.
Identity \eqref{uva} below will play a crucial role later on. 

\begin{lem}
   \label{lem:decay-eigenfunctions}
   Let $|\kappa|\ge \frac{1}{2}$. Let $u\not=0$ be an eigenfunction of $\mAk$.
   Then there exists a unique vector-valued function $q$ which is complex analytic in a neighbourhood of $[0,\pi]$ such that
   \begin{align} \label{uva}
      u(\theta) = 
      \theta^{|\kappa|}(\pi-\theta)^{|\kappa|} q(\theta).
   \end{align}
\end{lem}
\begin{proof}
Included in Appendix~\ref{app:frobenius}.
\end{proof}

This implies that every eigenfunction $u$ of $\mAk$ belongs to $\dom(\Dop)$ and $\dom(\Sopk)$ and therefore
\begin{align}
   \label{eq:sumrepresentation}
   \mAk u = \Dop u + \Sopk u + \pot u
\end{align}
where
\begin{align}
   \label{uchuva}
   \Dop =
   \begin{pmatrix}
      0 & \frac{\rd}{\rd\theta} \\
      -\frac{\rd}{\rd\theta} & 0
   \end{pmatrix},
   \qquad
   \Sopk =
   \begin{pmatrix}
      0 & \frac{\pi \kappa}{\theta(\pi-\theta)} \\
      \frac{\pi \kappa}{\theta(\pi-\theta)} & 0
   \end{pmatrix}
\end{align}

\begin{cor} \mbox{}
   \label{cor:H0} 
   Let $|\kappa|>\frac12$. Let $u$ be an eigenfunction of $\mAk$. The following properties hold true.
   \begin{enumerate} 
      \item \label{cor:H0a} $u\in H_0^1(0,\pi)$.      
      \item \label{cor:H0c}
      $u$ has a bounded $r$th derivative for every $r\in \mathbb{N}$ satisfying $1\leq r \leq |\kappa|$.
      \item \label{cor:H0b}
      $u\in H^r(0,\pi)$ for every $r< |\kappa|+\frac{1}{2}$.
   \end{enumerate}
\end{cor}

\begin{proof} \

   \emph{Statements \ref{cor:H0a} and \ref{cor:H0c}.}
   They follow directly from Lemma~\ref{lem:decay-eigenfunctions}.

   \emph{Statement \ref{cor:H0b}.}
   Let $\ell\in\mathbb{N}\cup\{0\}$ and $\epsilon\in(0,1]$ be such that $|\kappa|=\ell+\epsilon$.
   Let
   \[
   \tau(x)= \begin{cases}
       \sin(x) , & \text{if }\ \ell \equiv_{2} 0, \\
      \cos(x) , & \text{if }\ \ell \equiv_{2} 1.
   \end{cases}
   \]
Let
\begin{align}
   \label{pomelo}
   \widehat{u}_n=\sqrt{\frac{2}{\pi}} \int_0^\pi u(\theta) \tau( n \theta)\,\rd \theta,
\end{align}
be the Fourier coefficients of $u$ in $\tau(n\theta)$. Then $u\in H^r(0,\pi)$ if and only if 
\begin{equation}
\label{seaweed}
\sum_{n=1}^\infty \langle n \rangle^{2r}|\hat{u}_n|^2<\infty.
\end{equation}
Note that $n=0$ can be omitted here, because $u$ is always a bounded function. Let $r<|\kappa|+\frac{1}{2} = \ell +\epsilon+ \frac{1}{2}$ and choose $0<\epsilon'<\epsilon$ such that
$r < \ell +\epsilon'+ \frac{1}{2}$.
We show that \eqref{seaweed} holds true.

Integrating by parts $\ell + 2$ times in \eqref{pomelo} gives
\begin{align*}
   |\widehat{u}_n|
   %       &=\frac{1}{n^{\ell +1}} \sqrt{\frac{2}{\pi}} \left|\int_0^\pi u^{(\ell+1)}(\theta) \cos( n \theta)\,\rd \theta \right| \\
   &=\frac{1}{n^{\ell +2}}\sqrt{\frac{2}{\pi}}\left|\left[u^{(\ell+1)}(\theta) \sin( n \theta)\right]_{\theta=0}^{\theta=\pi} - \int_0^\pi u^{(\ell+2)}(\theta) \sin( n \theta)\,\rd \theta \right|.
\end{align*}
By virtue of \eqref{uva},
\[
u^{(\ell+1)}(\theta)=\theta^{-1 + \epsilon}(\pi-\theta)^{-1 + \epsilon} q_{\ell+1}(\theta)
\]
and
\[
u^{(\ell+2)}(\theta)=\theta^{-2 + \epsilon}(\pi-\theta)^{-2 + \epsilon} q_{\ell+2}(\theta)
\]
where the $q_{j}$ are analytic functions in a neighbourhood of $[0,\pi]$. Then, taking the limit as $\theta\to 0$ and $\theta\to \pi$, gives
\[
    \left[u^{(\ell+1)}(\theta) \sin( n \theta)\right]_{\theta=0}^{\theta=\pi}=0.
\]
Hence,
\begin{align*}
   |\widehat{u}_n|
   &= \frac{1}{n^{\ell +2}}\sqrt{\frac{2}{\pi}} \left|\int_0^\pi \theta^{-2 + \epsilon}(\pi-\theta)^{-2 + \epsilon} q_{\ell+2}(\theta) \sin( n \theta)\,\rd \theta \right|
   \\
   &\leq \frac{c_1}{n^{\ell +2}}
   \left(\max_{0\leq \theta \leq \pi}|\theta^{-1+\epsilon'}(\pi-\theta)^{-1+\epsilon'}\sin(n\theta)|
   \right)
   \int_0^\pi \theta^{-1 + \epsilon-\epsilon'}(\pi-\theta)^{-1 + \epsilon-\epsilon'}\,\rd \theta \\
   &= \frac{c_2}{n^{\ell +2}}  \left(\max_{0\leq \theta \leq \frac{\pi}{2}}|\theta^{-1+\epsilon'}\sin(n\theta)|
   \right)
   = \frac{c_2}{n^{\ell +2}}
   \max_{0\leq t \leq \frac{n\pi}{2}}\left|\left(\frac{t}{n}\right)^{-1+\epsilon'}\sin(t)\right| \\
   &\leq \frac{c_2}{n^{\ell +2}} n^{1-\epsilon'} \max_{0\leq t < \infty} \left|t^{-1+\epsilon'}\sin(t)\right| \leq  \frac{c_2}{n^{\ell +1+\epsilon'}}
\end{align*}
where $c_j$ are constants independent of $n$. Hence,
\begin{equation*}
   \sum_{n=1}^\infty \langle n \rangle^{2r} |\widehat{u}_n|^2  \leq c_2\sum_{n=1}^\infty \frac{\langle n \rangle^{2r}}{n^{2\ell+2+2\epsilon'}} .
\end{equation*}
As $2r-(2\ell+2+2\epsilon')<1$, the summation on the right hand side converges.
\end{proof}

\begin{rem} \label{parchita}
We believe that, whenever $|\kappa|>\frac12$ is not an integer, an optimal threshold for regularity is $u\in H^r(0,\pi)$ for all $r< |\kappa|+1$.
The proof of the latter may be achieved by interpolating the spectral projections of the operator $\mAk$ between suitable Sobolev spaces for $\kappa$ in an appropriate segment of the real line.
However, for the purpose of the linear interpolation setting presented below, this refinement is not essential.
\end{rem}

%%%}}}

\section{The second order spectrum and eigenvalue approximation}
\label{section:manzana}
%%%{{{

The self-adjoint operator $\mAk$ is strongly indefinite.
Therefore standard techniques such as the classical Galerkin method are not directly applicable for the numerical estimation of bounds for the eigenvalues of $\mAk$ due to variational collapse.  
As we shall see below, the computation of two-sided bounds for individual eigenvalues can be achieved by means of the quadratic method
\cite{1998Davies,2000Shargorodsky,2004Levitin}, which is convergent  \cite{MR2219033,MR2289273,2010Boulton} and is known to avoid spectral pollution completely.

Everywhere below we consider the simplest possible trial subspaces so that the discretisation of $\mAk$ is achievable in a few lines of computer code.
The various benchmark experiments reported in Section~\ref{platano} indicate that, remarkably, this simple choice already achieves a high degree of accuracy for the angular Kerr-Newman Dirac operator whenever $|\kappa|> \frac12$. 

Set $n\in \N$, $h = \pi/n$ and 
$\theta_j = j\pi/n$ for $j=0,\dots, n$.
Here and elsewhere below $\mathcal{L}_h$ denotes the trial subspace of continuous piecewise linear functions on $[0,\pi]$ with values in $\C^2$, vanishing at $0$ and $\pi$, such that their restrictions to the segments $[\theta_j,\, \theta_{j+1}]$ are affine. 
Without further mention we will always assume that $n\ge 4$, so that $0<h<1$. 

It is readily seen that $\mathcal{L}_h$ is a linear subspace of $\dom(\mAk)$ of dimension $2(n-1)$   and that
\[
     \mathcal{L}_h=\operatorname{Span}\left\{ \begin{bmatrix} b_j \\ 0 \end{bmatrix},  \,
     \begin{bmatrix} 0 \\ b_j \end{bmatrix} \right \}_{j=1}^{n-1}
\] 
where
\[
    b_j(\theta)=
    \left\{ \begin{array}{ll} \frac{\theta-\theta_{j-1}}{\theta_j-\theta_{j-1}},
   & \theta_{j-1}\leq \theta\leq \theta_j, \\[1ex]
   \frac{\theta_{j+1}-\theta}{\theta_{j+1}-\theta_{j}}, &
   \theta_{j}\leq \theta\leq \theta_{j+1}, \\[1ex]
   0, & \text{otherwise}, \end{array}
   \right. \qquad j=1,\ldots,n-1.
\]
For any given $u\in H^1(0,\pi)$,  $u_h\in\mathcal L_h$ will be the unique (nodal) interpolant which satisfies
\[
      u_h(\theta_j)=u(\theta_j), \qquad  j=1,\ldots,n-1,
\]
that is
\[
      u_h(\theta)=\sum_{j=1}^{n-1}  b_j(\theta) \begin{bmatrix} u_1(\theta_j) \\u_2(\theta_j)  \end{bmatrix}.
\]

Set 
\begin{align*}
Q^h&=[Q^h_{jk}]_{j,k=1}^{n-1}, &Q^h_{jk}&=\begin{bmatrix}  \left\langle\mAk\begin{bmatrix} b_j \\ 0 \end{bmatrix},\mAk\begin{bmatrix} b_k \\ 0 \end{bmatrix}\right\rangle &
\left\langle\mAk\begin{bmatrix} b_j \\ 0 \end{bmatrix},\mAk\begin{bmatrix}  0 \\ b_k \end{bmatrix}\right\rangle \\ \\
\left\langle\mAk\begin{bmatrix}  0 \\ b_j \end{bmatrix},\mAk\begin{bmatrix} b_k \\ 0 \end{bmatrix}\right\rangle &
\left\langle\mAk\begin{bmatrix}  0 \\ b_j \end{bmatrix},\mAk\begin{bmatrix} 0 \\ b_k \end{bmatrix}\right\rangle   \end{bmatrix},
\\ \\
R^h&=[R^h_{jk}]_{j,k=1}^{n-1}, &R^h_{jk}&=\begin{bmatrix}  \left\langle\mAk\begin{bmatrix} b_j \\ 0 \end{bmatrix},\begin{bmatrix} b_k \\ 0 \end{bmatrix}\right\rangle &
\left\langle\mAk\begin{bmatrix} b_j \\ 0 \end{bmatrix},\begin{bmatrix}  0 \\ b_k \end{bmatrix}\right\rangle \\ \\
\left\langle\mAk\begin{bmatrix}  0 \\ b_j \end{bmatrix},\begin{bmatrix} b_k \\ 0 \end{bmatrix}\right\rangle &
\left\langle\mAk\begin{bmatrix}  0 \\ b_j \end{bmatrix},\begin{bmatrix} 0 \\ b_k \end{bmatrix}\right\rangle   \end{bmatrix},
\\ \\
S^h&=[S^h_{jk}]_{j,k=1}^{n-1}, &S^h_{jk}&=\begin{bmatrix}  \left\langle\begin{bmatrix} b_j \\ 0 \end{bmatrix},\begin{bmatrix} b_k \\ 0 \end{bmatrix}\right\rangle &
\left\langle\begin{bmatrix} b_j \\ 0 \end{bmatrix},\begin{bmatrix}  0 \\ b_k \end{bmatrix}\right\rangle \\ \\
\left\langle\begin{bmatrix}  0 \\ b_j \end{bmatrix},\begin{bmatrix} b_k \\ 0 \end{bmatrix}\right\rangle &
\left\langle\begin{bmatrix}  0 \\ b_j \end{bmatrix},\begin{bmatrix} 0 \\ b_k \end{bmatrix}\right\rangle   \end{bmatrix}.
\end{align*}
These are the $2(n-1)\times 2(n-1)$ bending, stiffness and mass matrices, associated to  $\mAk$ for
the trial subspace $\mathcal{L}_h$. A complex number $z$ is said to belong to the second order spectrum of $\mAk$ relative to $\mathcal{L}_h$, $\spec_2(\mAk,\mathcal{L}_h)$, if and only if there exists a non-zero
$\underline{u}\in \mathbb{C}^{2(n-1)}$ such that
\[
(Q^h-2zR^h+z^2S^h)\underline{u} = 0 .
\]
All the matrix coefficients of this quadratic matrix polynomial are hermitian, therefore the non-real points in  $\spec_2(\mAk,\mathcal{L}_h)$ always form conjugate pairs.

For $a<b$ denote by
\[
   \mathbb{D}(a,b)=\left\{z\in \mathbb{C}:\left| z-\frac{a+b}{2} \right|<\frac{b-a}{2} \right\}  
\]
the open disk whose diameter is the segment $(a,b)$.
The following crucial connection between the second order spectra and the spectrum allows computation of numerical bounds for the eigenvalues of $\mAk$.
See \cite{2000Shargorodsky} or \cite{2004Levitin}, also \cite[Lemma~2.3]{2010Boulton}.

\begin{lem} \label{limon}
    If $(a,b)\cap \spec(\mAk)=\varnothing$, then $\mathbb{D}(a,b)\cap \spec_2(\mAk,\mathcal{L}_h)=\varnothing$.
\end{lem}

A first crucial consequence of this lemma is that
 \begin{equation}    \label{guayaba}
z\in \spec_2(\mAk,\mathcal{L}_h) \quad \Longrightarrow \quad    [\Re(z)-|\Im(z)|,\, \ \Re(z)+|\Im(z)|] \cap  \spec(\mAk) \neq \varnothing.
 \end{equation}
That is, segments centred at the real part of conjugate pairs in the second order spectrum are guaranteed intervals of enclosure for the eigenvalues of $\mAk$.

Moreover, if we possess rough \emph{a priori} certified information about the position of the eigenvalues of $\mAk$, the enclosure can be improved substantially.
To be precise,
\begin{equation}   \label{lulo}
   \left.
   \begin{array}{r}
      (a,b)\cap  \spec(\mAk)=\{\lambda\} \\[1ex]
      z\in \mathbb{D}(a,b)
   \end{array}
   \right\}
   \quad \Longrightarrow \quad
   \Re(z)-\frac{|\Im(z)|^2}{b-\Re(z)}< \lambda < \Re(z)+\frac{|\Im(z)|^2}{\Re(z)-a}.
\end{equation}
See \cite{2006Boulton,2010Strauss}.

Both \eqref{guayaba} and \eqref{lulo} will be employed for concrete calculations in Section~\ref{platano}.
The segment in \eqref{lulo} will have a smaller length than that in \eqref{guayaba} only if $z\in\spec(\mAk,\mc{L}_h)$ is
very close to the real line.
As we shall see next, this will be ensured if the angle between $\ker(\mAk-\lambda)$ and $\mc{L}_h$ is small.
For a proof of this technical statement see \cite[Theorem~3.2]{2014Boulton} and \cite{HobinyThesis}.
See also \cite{2010Boulton}.
Recall that all the eigenvalues are simple, Remark~\ref{naranja}.

\begin{lem} \label{ciruela}     
Let $u\in \ker(\mAk-\lambda)$ be such that $\|u\|=1$.
There exist constants $K>0$ and $\epsilon_0>0$ ensuring the following.  If
\begin{equation} \label{papaya}
  \min_{v\in \mc{L}_h} \left( \|u-v\|+\|\mAk( u-v)\|    \right)< \epsilon 
\end{equation}
for some $\epsilon\leq \epsilon_0$, then we can always find $\lambda_h\in \spec_2(\mAk,\mc{L}_h)$ such that
\begin{equation} \label{cambur}
  |\lambda_h-\lambda|<K\epsilon^{1/2}.
\end{equation}
\end{lem}

As it has been observed previously in \cite{2010Boulton} and \cite{MR2219033}, most likely the term $\epsilon^{1/2}$ in \eqref{cambur} can be improved to $\epsilon^{1}$.
The results of our numerical experiments are in agreement with this conjecture, see Section~\ref{fivepointfour} and Figure~\ref{lima}.
\smallskip

A concrete estimate on the convergence of
the second order spectra to the real line, and hence the spectrum, follows.

\begin{thm} \label{mamoncillo}
 Let $|\kappa|>1/2$. Fix $0<r<|\kappa|-\frac12$ and let  
 \[
 p(\kappa)=\begin{cases}
    r, & |\kappa|\in \left(\frac12,\frac32\right]\setminus\{1\}, \\
    1, & \text{otherwise}. 
 \end{cases} 
 \]
 Let $\lambda \in \spec (\mAk)$. There exist constants $h_0>0$ and $K>0$ such that
 \begin{equation} \label{cemeruco}
  |\lambda_h-\lambda|< K h^{\frac12 p(\kappa)},   \qquad  0<h<h_0,
 \end{equation}
for some $ \lambda_h \in \spec_2(\mAk,\mc{L}_h)$.
\end{thm}

The proof of this statement is presented separately in the next section. Roughly speaking it reduces to finding suitable estimates for the left hand side of \eqref{papaya} from specific estimates on the residual in the piecewise linear interpolation of the eigenfunctions of $\mAk$.
These estimates are of the order $h^{p(\kappa)}$, so that a direct application of Lemma~\ref{ciruela} will lead to the desired conclusion. See Section~\ref{fivepointfour}.

\begin{rem}
An explicit expression for $K$ can be determined by examining closely the proof of
\cite[Theorem~3.2]{2014Boulton} and carefully following track of the different constants that will appear in Section~\ref{section:ceresa}.
\end{rem}

%%%}}}

\section{The proof of Theorem~\ref{mamoncillo}}
\label{section:ceresa}
%%%{{{

The next inequalities are standard in the theory of piecewise linear interpolation of functions in one dimension,
\cite[Remark 1.6 and Proposition 1.5]{ErnGuermond}:
% \cite[Theorem 3.2.1]{Ciarlet}:
\begin{alignat}{2}
   \label{eq:IdEstimateSmall}
   \| u - u_h \|  &\le  h \|u'\|
   &&\quad\text{ for } u\in H^1(0,\pi),
   \\
   \label{eq:IdEstimateLarge}
   \| u - u_h \|  &\le  h^2 \|u''\|
   &&\quad\text{ for } u\in H^2(0,\pi),
   \\
   \label{eq:IdEstimateDog} 
   \| (u - u_h)' \|  &\le  h \|u''\|
   &&\quad\text{ for } u\in H^2(0,\pi).
\end{alignat}
We will employ these identities below, as well as the inequality:
\begin{alignat}{2}
      \label{eq:H1derivative}
       \|(u-u_h)'\|&\le 2 \|u'\|&&\quad\text{ for } u\in H^1(0,\pi).
   \end{alignat}

The proof of \eqref{eq:H1derivative} can be achieved as follows. Let $u\in H^1(0,\pi)$. Since $(u_h)'$ is constant along $(\theta_j, \theta_{j+1})$ and
    $u'(\theta) = \frac{1}{h}(u(\theta_{j+1}) -  u(\theta_j))$ for every $\theta\in(\theta_j, \theta_{j+1})$, then
   \begin{align*}
      \int_{\theta_j}^{\theta_{j+1}} |(u_h)'(\theta)|^2\,\rd \theta
      & = h^{-1} \left| u_h(\theta_{j+1}) - u_h(\theta_j) \right|^2
      = h^{-1} \left| u(\theta_{j+1}) - u(\theta_j) \right|^2
      \\
      &= h^{-1} \left| \int_{\theta_j}^{\theta_{j+1}} u'(\theta)\,\rd \theta\right|^2
      \le \int_{\theta_j}^{\theta_{j+1}} |u'(\theta)|^2 \,\rd \theta.
   \end{align*}
In the last step we invoke H\"older's inequality.
Adding up each side for $j$ from $0$ to $n-1$ and then taking the square root gives 
\begin{equation*}   
\|(u_h)'\| \le \|u'\|.
\end{equation*} 
By virtue of the triangle inequality, \eqref{eq:H1derivative} follows.

\begin{lem}
   \label{lem:DkEstimate}
   Let $\Dop$ be as in \eqref{uchuva}. Let $\alpha\in[0,1]$. 
   If $u\in H^{1+\alpha}(0,\pi)$ and $u_h$ is its nodal interpolant, then
   \begin{align}
   \label{papayuela}
   \| \Dop(u-u_h)\| &\le 2^{1-\alpha}h^{\alpha}\|u\|_{1+\alpha}.
   \end{align}
   \end{lem}
\begin{proof}
   By virtue of \eqref{eq:H1derivative}, for all $u\in  H^1(0,\pi) $,
   \begin{align}
      \label{eq:DkEstimateSmall}
      \| \Dop(u-u_h) \|
      &= \left\|
      \begin{pmatrix}
	 0 & \frac{\rd}{\rd\theta} \\ 
	 -\frac{\rd}{\rd\theta} & 0
      \end{pmatrix}
      (u - u_h)
      \right\|
      \le
      \left\|
      \begin{pmatrix}
	 0 & 1 \\ -1 & 0
      \end{pmatrix}
      \right\|
      \left\|
	 (u - u_{h})'
      \right\|\leq 2 \|u'\|.
    \end{align}
   If additionally $u\in H^2(0,\pi)$, then the same argument combined with \eqref{eq:IdEstimateDog} yields 
   \begin{equation}
      \label{eq:DkEstimateLarge}
      \| \Dop(u - u_h) \|  \le  h \|u''\|.
   \end{equation}

   Define the linear operators
   \begin{alignat*}{3}
      &T_1: H^1(0,\pi) \to L_2(0,\pi),
      \quad
      &&u\mapsto \Dop(u-u_h),
      \\
      &T_2: H^2(0,\pi) \to L_2(0,\pi),
      \quad
      &&u\mapsto \Dop(u-u_h).
   \end{alignat*}
   From \eqref{eq:DkEstimateSmall} and \eqref{eq:DkEstimateLarge}, it follows that
   $\|T_1\|\le 2$ and $\|T_2\|\le h$.
   Hence, by complex interpolation, for every $0\le\alpha\le1$
   \begin{align*}
      T_{1+\alpha}: H^{1+\alpha}(0,1) \to L_2(0,1),
      \quad
      u\mapsto\Dop(u-u_h)
   \end{align*}
   is a bounded operator with norm $\|T_{1+\alpha}\| \le 2^{1-\alpha} h^{\alpha}$.
   See \cite[Appendix~to~IX.4 and Problem~36]{ReedSimonII1980}.
\end{proof} 

Recall the decomposition of $\mAk u$ given in \eqref{eq:sumrepresentation} and the representation of the eigenfunctions in \eqref{uva}.

\begin{lem}
   \label{lem:SkEstimate}
Let $|\kappa|> \frac{1}{2}$. Let $u$ be an eigenfunction of $\mAk$ and let $q$
be related to $u$ by means of \eqref{uva}. 
Set
   \begin{align*}
      d_1(u)
      &= 
      \sqrt{2} |\kappa|
      \left[
      4 \pi^{2|\kappa|}
      \frac{2|\kappa|}{2|\kappa|-1} 
            \max_{0\le\theta\le \pi}|q(\theta)|^2
      +
      \frac14 \max_{0\le\theta\le \pi}|u''|^2
      \right]^{\frac{1}{2}},
      \\[1ex]
      d_2(u)
      & = 
      \sqrt{2} |\kappa|
      \left[
      4 \pi^{2|\kappa|}
       \frac{2|\kappa|}{2|\kappa|-1} 
             \max_{0\le\theta\le \pi} |q(\theta)|^2
      +
      \frac{b(u)^2( 6-2|\kappa|)}{4(5-2|\kappa|)}
      \right]^{\frac{1}{2}},
      \\[1ex]
      b(u)&=
      \max_{0\le\theta\le \pi/2} \frac{|u''(\theta)|}{\theta^{|\kappa|-2}} +
      \max_{\pi/2\le\theta\le \pi} \frac{|u''(\theta)|}{(\pi-\theta)^{|\kappa|-2}}.
  \end{align*}      
  Then   \begin{alignat}{2}
      \label{eq:SkEstimateLarge}
      &\| \Sopk (u - u_h) \|  \le  d_1(u) h^{ \frac{3}{2} },
      &\qquad& |\kappa| \geq 2 \quad \text{or} \quad |\kappa|=1,
      \\
      \label{eq:SkEstimateSmall}
      &\| \Sopk (u - u_h) \|  \le  d_2(u) h^{|\kappa|-\frac{1}{2}},
      && \frac{1}{2} < |\kappa| \leq 2.
   \end{alignat}

 \end{lem}
\begin{proof} %%%{{{
Firstly observe that
   \begin{align}
      \nonumber
      \| \Sopk (u - u_h) \|^2
      &= 
      \int_0^\pi \Bigg( \frac{\pi \kappa}{\theta(\pi-\theta)} \Bigg)^2 \, 
      \left|
      \begin{pmatrix} 0 & 1 \\ -1 & 0
      \end{pmatrix}
      (u - u_h)(\theta)
      \right|^2\,\rd \theta
      \\
      \nonumber
      &\le
      \pi^2 \kappa^2 \int_0^\pi (\theta(\pi-\theta))^{-2} | (u - u_h)(\theta)|^2 \,\rd \theta
      \\
      \label{eq:estimates}
      &=
      \pi^2 \kappa^2
      \left(
      \integral_1 + \integral_2 + \integral_3 + \integral_4
      \right)
   \end{align}
   where
   \begin{align*}
      \integral_\ell =
      \int_{\alpha_\ell}^{\alpha_{\ell+1}}
      (\theta(\pi-\theta))^{-2} | (u - u_h)(\theta)|^2 \,\rd \theta
   \end{align*}
   for $(\alpha_\ell)_{\ell=0}^4 = (0,\, h,\, \frac{\pi}{2},\, \pi-h, \pi)$.

The interpolant $u_h$ has the form $u_h(\theta) = \frac{\theta}{h} u(h) = \theta h^{|\kappa|-1}(\pi-h)^{|\kappa|} q(h)$ for $\theta\in[0,h]$. 
Therefore
   \begin{align*}
	 \integral_1
	 &=
	 \int_0^h \theta^{-2}(\pi-\theta)^{-2} \left| \theta^{|\kappa|}(\pi-\theta)^{|\kappa|}q(\theta) - \theta h^{|\kappa|-1}(\pi-h)^{|\kappa|} q(h)\right|^2 \,\rd \theta
	 \\
	 &\leq
	 2\int_0^h 
	 \theta^{2|\kappa|-2}(\pi-\theta)^{2|\kappa|-2}|q(\theta)|^2
	 + h^{2|\kappa|-2}(\pi-h)^{2|\kappa|}(\pi-\theta)^{-2} |q(h)|^2 
	 \,\rd \theta
	 \\
	 &\leq
	 \frac{2h^{2|\kappa|-1}}{{2|\kappa|-1}}
	 \max_{0\le\theta\le h}\left\{ (\pi-\theta)^{2|\kappa|-2}|q(\theta)|^2\right\}
	 + 2 h^{2|\kappa|-1} (\pi-h)^{2|\kappa|-2} |q(h)|^2 
	 \\
	 &\leq
	 2h^{2|\kappa|-1}
	 \max_{0\le\theta\le \pi/2}\left\{ (\pi-\theta)^{2|\kappa|-2}|q(\theta)|^2\right\}
	 \left( \frac{1}{{2|\kappa|-1}} + 1 \right).
   \end{align*}
Now
\[(\pi-\theta)^{2|\kappa|-2}
=(\pi-\theta)^{2|\kappa|-1} (\pi-\theta)^{-1}
\le \pi^{2|\kappa|-1} \frac{2}{\pi}
= 2\pi^{2|\kappa|-2}\quad
\text{for any }\quad 0\le \theta \le \pi/2.\]
Thus, setting 
\begin{equation}
c_1(u) = 4 \pi^{2|\kappa|-2}
\left( \frac{2|\kappa|}{{2|\kappa|-1}}  \right)
\max_{0\le\theta\le \pi}\left\{ |q(\theta)|^2\right\},
\end{equation}
gives
   \begin{equation}
      \label{eq:zeroh}
	 \integral_1 \le h^{2|\kappa|-1} c_1(u).
   \end{equation}
Analogously one can show
   \begin{equation}
      \label{eq:zeropi}
	 \integral_4 \le h^{2|\kappa|-1} c_1(u).
   \end{equation}

   Now we estimate $\integral_2$ and $\integral_3$.
   Note that $u$ is $C^\infty$ in the open interval $(0,\pi)$ and for $1\le j <k \le n/2$ we have
   \begin{align}
      \label{eq:secondderiv1}
      | u(\theta) - u_h(\theta) | 
      \le \frac{\sqrt{2}\, h^2}{8}\ \max\Big\{ | u''(\theta) | : \theta\in [\theta_j,\,\theta_{k}] \Big\},
      \quad \theta\in[\theta_j,\,\theta_{k}].
   \end{align}
   
   First assume that $|\kappa| \ge 2$ (or $|\kappa|=1$).
   According to Corollary~{\ref{cor:H0}\ref{cor:H0c}} (or Lemma~\ref{lem:decay-eigenfunctions}), $u$ has a bounded second derivative and therefore \eqref{eq:secondderiv1} yields
   $| u(\theta) - u_h(\theta) | \le \frac{\sqrt{2}h^2}{8}\ \max_{\theta\in[0,\pi]}| u''|$ in $[0,\pi]$.
   So
   \begin{align*}
      \integral_2&
      \le \frac{h^4}{32} \max_{\theta\in[0,\pi]}|u''|^2\
      \int_{h}^{\frac{\pi}{2}} (\theta(\pi-\theta))^{-2}\,\rd \theta \\ 
      &\le \frac{h^4}{8\pi^2} \max_{\theta\in[0,\pi]}|u''|^2\
      \int_{h}^{\frac{\pi}{2}} \theta^{-2}\,\rd \theta
      \le \frac{h^3}{8\pi^2} \max_{\theta\in[0,\pi]}|u''|^2.
   \end{align*}
   If we perform the analogous calculations for $\integral_3$, we conclude
   \begin{equation}
   \label{eq:I23small}
   \integral_2 \le c_2(u) h^3, \qquad
   \integral_3 \le c_2(u) h^3
   \qquad\qquad
   \text{for }\ |\kappa| > 2
   \end{equation}
   where
   $c_2(u) = \frac{1}{8\pi^2}\max_{\theta\in[0,\pi]}|u''|^2$.

   For $\frac{1}{2}  \leq |\kappa| < 2$ (except for the case $|\kappa|=1$), the second derivative of $u$ diverges as $\theta\to 0$ of order $|\kappa|-2$  and the calculations above cannot be performed.
   However, $u''$ is analytic in $(0,\pi)$ because
   \begin{align*}
      u''(\theta) &=
      \theta^{|\kappa|-2} 
      \left[
      |\kappa|(|\kappa|-1)(\pi-\theta)^{|\kappa|} q(\theta) + 2|\kappa| \theta [ (\pi-\theta)^{|\kappa|} q(\theta)]' + \theta^2 [ (\pi-\theta)^{|\kappa|} q(\theta)]''
      \right].
   \end{align*}
Hence
   \begin{align*}
      |u''(\theta)| &\le
      b(u) \theta^{|\kappa|-2},
      \qquad \theta\in (0,\textstyle\frac{\pi}{2}].
   \end{align*}
   Since $|\kappa|-2 \leq 0$, the function $\theta\mapsto \theta^{2|\kappa|-4}$ is positive non-increasing in the segment $[h,\, \frac{\pi}{2}]$.
   Therefore, from \eqref{eq:secondderiv1} we estimate $\integral_2$ as follows.
   \begin{align*}
      \integral_2
      &= \sum_{j=1}^{n/2 -1}
      \int_{\theta_j}^{\theta_{j+1}}
      \theta^{-2} (\pi-\theta)^{-2} 
      | (u - u_h)(\theta) |^2
      \,\rd \theta
      \\
      &\le
      \frac{h^4}{32} 
      \sum_{j=1}^{n/2 -1}
      \int_{\theta_j}^{\theta_{j+1}}
      (\theta(\pi-\theta))^{-2} \max_{\theta\in[\theta_j,\, \theta_{j+1}]}\big\{ | u''(\theta)|^2  \big\}
      \ \rd \theta
      \\
      &\le
      \frac{b(u)^2 h^4}{32}\ \ \sum_{j=1}^{n/2-1}
      \max_{\theta\in[\theta_j,\, \theta_{j+1}]}
      \big\{ \theta^{ 2|\kappa|-4 } \big\}
      \int_{\theta_j}^{\theta_{j+1}}
      \theta^{-2} 
      (\pi-\theta)^{-2} 
      \,\rd \theta
      \\
      &\le
      \frac{b(u)^2 h^4}{8\pi^2}\ \ \sum_{j=1}^{n/2-1}
      \theta_{j}^{2|\kappa|-4}
      \Big(\theta_{j}^{-1} - \theta_{j+1}^{-1} \Big)
      \\
      &=
      \frac{b(u)^2 h^4}{8\pi^2}\ \ \sum_{j=1}^{n/2-1}
      (jh)^{2|\kappa|-4}
      \frac{1}{hj(j+1)}
      \\
      &\le
      \frac{b(u)^2 h^{2|\kappa|-1}}{8\pi^2}\ \ \sum_{j=1}^{n/2-1}
      j^{2|\kappa|-6}
      \le
      \frac{b(u)^2 h^{2|\kappa|-1}}{8\pi^2}\ \ 
      \left[ 1 + \int_{1}^{\infty} t^{2|\kappa|-6}\,\rd t \right]
      \\
      & =
      \frac{b(u)^2 h^{2|\kappa|-1}}{8\pi^2}
      \left[
      1+ \frac{ 1 }{5-2|\kappa|}
      \right].
   \end{align*}
   Set $\widetilde c_2(u) = \frac{b(u)^2}{8\pi^2} \left[ 1+ \frac{ 1 }{5-2|\kappa|} \right]$.
   Then, performing similar computations for $\integral_3$, we obtain
   \begin{equation}
   \label{eq:I23large}
   \integral_2 \le \widetilde c_2(u) h^{2|\kappa|-1},
   \qquad
   \integral_3 \le \widetilde c_2(u) h^{2|\kappa|-1}
   \qquad
   \qquad
   \text{for }\ \frac12< |\kappa| \le 2. 
   \end{equation}

Inserting
\eqref{eq:zeroh},
\eqref{eq:zeropi},
\eqref{eq:I23small} and
\eqref{eq:I23large} respectively
into \eqref{eq:estimates} yields
\begin{alignat*}{2}
\| \Sopk (u - u_h) \| &\le \pi^2\kappa^2 \sqrt{ 2c_1(u)h^{2|\kappa|-1} + 2\widetilde c_2(u)h^{2|\kappa|-1} } \le d_1(u) h^{|\kappa|-1/2}
\intertext
{for $ \frac12 < |\kappa| \le 2$, and}
\| \Sopk (u - u_h) \| &\le \pi^2\kappa^2 \sqrt{ 2c_1(u)h^{2|\kappa|-1} + 2c_2(u)h^{3} } 
\le d_2(u) h^{3/2}
\end{alignat*}
for $|\kappa| > 2 \text{ or } |\kappa|=1$.
\end{proof}
%%%}}}

% \bigskip

   \begin{table}
\renewcommand{\arraystretch}{1.5}
   \begin{tabular}{l|c|c|c|c|c}
   \rule[-2ex]{0cm}{5ex}
   & $\frac12 < |\kappa| < 1$ & 1 &$1< |\kappa| \le \frac32$ & $\frac32 < |\kappa| \le 2$ & $|\kappa| > 2$
   \\ \hline
   $\|(u-u_h)\|$            & $1$ & $2$ & $1$ & $2$ & $2$
   \\
   $\|\Sopk(u-u_h)\|$       & $|\kappa|- \frac12$ & $\frac32$ &$|\kappa|- \frac12$ & $|\kappa|- \frac12$ & $\frac32$
   \\
   $\|\Dop_\kappa(u-u_h)\|$ & $\cellcolor{mygrey}{<|\kappa|-\frac12}$ &$\cellcolor{mygrey}{1}$ & $\cellcolor{mygrey}{< |\kappa|-\frac12}$ & $\cellcolor{mygrey}{1}$ & $\cellcolor{mygrey}{1}$ 
   \end{tabular}
   \bigskip
   \bigskip
   \caption{
   A summary of the different estimates for the powers of $h$ employed in the proof of
   Theorem~\ref{thm:main}. See
   \eqref{eq:IdEstimateSmall}, \eqref{eq:IdEstimateLarge},
   \eqref{papayuela},
   \eqref{eq:SkEstimateLarge} and \eqref{eq:SkEstimateSmall}.
   Also Corollary~\ref{cor:H0}.
   The term of lowest order is shaded.}
   \label{summary}
\end{table}

The next statement ensures the validity of Theorem~\ref{mamoncillo}.

\begin{thm}
   \label{thm:main}
   Let $|\kappa|>\frac{1}{2}$. Let $\lambda \in\spec(\mAk)$ and $u\in\dom(\mAk)$ be an eigenpair for $\mAk$.
   Assume that $\|u\|=1$.
   Then there exists a constant $c>0$ ensuring the following. For every $h>0$,
   \begin{align}
      \label{eq:estimateLargeKappa}
      &\| u - u_h \| + \| \mAk u - \mAk u_h \| \le  c h,
      &&\quad |\kappa| > \frac32 \quad \text{or} \quad |\kappa|=1 ,
      \\
      \label{eq:estimateSmallKappa}
      &\| u - u_h \| + \| \mAk u - \mAk u_h \| \le  c h^{r},
      &&\quad |\kappa| \in \left(\frac{1}{2},\, \frac{3}{2}\right] \setminus\{ 1 \},\ r< |\kappa|-\frac{1}{2}.
   \end{align}
\end{thm}
\begin{proof}
   %%%{{{
   Using \eqref{eq:sumrepresentation}, we obtain
   \begin{align*}
      \| u - u_h \| + \| \mAk u - \mAk u_h \| 
      \le
      (1+\|\pot\|) \| u - u_h \| + 
      \| \Dop(u - u_h) \| + 
      \| \Sopk(u - u_h) \|.
   \end{align*}

   Part~\ref{cor:H0b} of Corollary~\ref{cor:H0} ensures that
   $u\in H^2(0,\pi)$ if $|\kappa| > \frac32$, and
   $u\in H^r(0,\pi)$ for any $r< |\kappa| + \frac12$ if $\frac12 < |\kappa| \le \frac32$.
   The statements \eqref{eq:estimateLargeKappa} and \eqref{eq:estimateSmallKappa} follow from 
   \eqref{eq:IdEstimateSmall}, \eqref{eq:IdEstimateLarge},
   Lemma~\ref{lem:DkEstimate} and Lemma~\ref{lem:SkEstimate}. See Table~\ref{summary}.
   %%%}}}
\end{proof}

%%%}}}

\section{Numerical benchmarks}    \label{platano}
%%%{{{

We now determine various numerical approximations of intervals of enclosure for the eigenvalues of the angular Kerr-Newman Dirac  operator~\eqref{frambuesa} by means of suitable combinations of \eqref{guayaba} and \eqref{lulo}. 
In order to  implement \eqref{lulo}, we consider the analytic enclosures derived in \cite{WinklmeierPhD} and \cite{Win08}.
Our purpose here is twofold.
On the one hand we verify the numerical quantities reported in \cite{SFC} and \cite{chakrabarti}.
On the other hand we establish new sharp benchmarks for the eigenvalues of $\mAk$.

Denote  by 
$\lambda_n\equiv \lambda_n(\kappa; am, a\omega)$,
where
\[
   -\infty< \dots <\lambda_{-n}<\ldots<\lambda_{-1}<0\leq
   \lambda_1<\ldots<\lambda_n< \dots < \infty,
\]
the eigenvalues of $\mAk$ for potential given by \eqref{coco}.
Explicit expressions for these eigenvalues are known only if $am=\pm a\omega$.
In this case,
\begin{equation}
   \label{eq:lambdatau}
   \lambda_n(\kappa; am, \pm am)
   = \pm \frac12 + \sign(n)\sqrt{ \left( \lambda_n(\kappa, 0,0) \mp \frac12 \right)^2 \pm 2\kappa am + (am)^2 }
\end{equation}
where
\begin{equation*}
   \lambda_n(\kappa; 0,0) = \sign(n) \Big( |\kappa|-\frac{1}{2} + |n| \Big),
   \qquad
   n\in\Z\setminus\{0\}.
\end{equation*}
See
\cite[Formula (45)]{BSW2005}.
For $am\not =\pm a\omega$, the two canonical references on numerical approximations of $\lambda_n(\kappa; am, a\omega)$ are \cite{SFC} and \cite{chakrabarti}.

Suffern \emph{et al} derived in \cite{SFC}  an asymptotic expansion of the form \[\lambda_n = \sum_{r,s} C^n_{r,s} (m-\omega)^r (m+\omega)^s.\]
The coefficients $C^n_{r,s}$ can be determined from a suitable series expansion of the eigenfunctions in terms of hypergeometric functions.
On the other hand, Chakrabarti  \cite{chakrabarti} wrote the eigenfunctions in terms of spin weighted spherical harmonics  and derived an expression for the squares of the eigenvalues in terms of $a\omega$ and $\omega/m$.
The tables reported in \cite[Tables~1-3]{chakrabarti} include predictions for the values of $\lambda_{-1}^2$ and $\lambda_{-2}^2$ for various ranges of $\kappa$, $a\omega$ and $am$.
It has been shown (\cite[Formula (45) and Remark 2]{BSW2005}) that \cite[Formula~(54)]{chakrabarti} and \eqref{eq:lambdatau} differ in the case $a\omega = am$, and that the correct expression turns out to be the latter.
See tables~\ref{table:chakra3by2p} and \ref{table:chakra3by2m} below.

In both \cite{SFC} and \cite{chakrabarti}, the numerical estimation of $\lambda_n$ is achieved by means of a series expansions in terms of certain expressions involving $a\omega$ and $am$,
so it is to be expected that the approximations in both cases become less accurate as $|a\omega|$ and $|am|$ increase.
However, no explicit error bounds are given in these papers and they seem to be quite difficult to derive.

A computer code written in Comsol LiveLink v4.3b, which we developed in order to produce all the computations reported here, is available in Appendix~\ref{mandarina}.
In all the calculations reported here the relative tolerance of the eigenvalue solver and integrators was set to $10^{-12}$.

\subsection{The paper \upshape{\protect\cite{chakrabarti}} }

Our first experiment consists in assessing the quality of the  numerical approximations reported in \cite[Table~2b]{chakrabarti} for $a\omega \not= am$, by means of a direct application of \eqref{guayaba}.
For this purpose we fix $h=0.001$. 
 
The tables~\ref{table:chakra3by2p} and \ref{table:chakra3by2m} contain computations of $|\lambda_{-1}(\pm 3/2,am,a\omega)|$ for the range
\[
 a\omega\in\{0.1, 0.2, \dots, 1.0\}, \qquad \omega/m  \in \{0,\, 0.1,\, \dots,\, 1.0\}.
\]
On the top of each row we have reproduced the positive square root of the original numbers from \cite[Table~2b]{chakrabarti}.
On the bottom of each row, we show  the corresponding correct eigenvalue enclosures with upper and lower bounds displayed in small font.
These bounds were obtained from \eqref{guayaba}, by computing the conjugate pairs $z,\overline{z}\in \spec_2(\mAk,\mathcal{L}_{0.001})$  near the segment $(-3,3)$.

Only for $a\omega=0.1,\,0.2,\,0.3$ (and the pair $(a\omega,m/\omega)=(0.4,0)$ when $\kappa=-3/2$), the predictions made in \cite{chakrabarti} are inside the certified enclosures.
For $\kappa=3/2$ they are always above the corresponding enclosure and for $\kappa=-3/2$ they are always below it. 
We have highlighted the relative degree of disagreement with the other quantities in different shades of colour.

\subsection{The paper \upshape{\protect\cite{SFC}} and sharp eigenvalue enclosures}
In this next experiment we validate the numbers reported in \cite{SFC} by means of sharpened eigenvalue enclosures determined from \eqref{lulo}.
This requires knowing beforehand some rough information about the position of the eigenvalues and the neighbouring spectrum. 
For this purpose, we have employed a combination of the analytical inclusions found in \cite{WinklmeierPhD} and \cite{Win08}, and numerically calculated inclusions determined from \eqref{guayaba}.
This technique allows reducing by roughly two orders of magnitude the length of the segments of eigenvalue inclusion. 

The columns in Table~\ref{tab:SFC} marked as ``A'' are analytic upper and lower bounds for the eigenvalues calculated following
\cite[Theorem 4.5]{Win08} and 
\cite[Remarks 6.4 and 6.5]{WinklmeierPhD}.
For our choices of the physical parameters, we always find that the upper bound for the $n$th eigenvalue is less than the lower bound for the $(n+1)$th eigenvalue, so each one of these segments contains a single  non-degenerate  eigenvalue of $\mAk$.
The columns marked as ``N'' were determined by fixing $h=0.001$ and applying directly \eqref{guayaba} in a similar fashion as for the previous experiment.
When these are contained in the former, which is not always the case (see the rows corresponding to $\kappa=\pm \frac12$ for $am=0.005$ and $aw=0.015$), it is guaranteed that there is exactly one eigenvalue in each one of these smaller segments.

\begin{rem}\label{sandia}
   The approach employed in \cite{WinklmeierPhD} and \cite{Win08}
   involves a perturbation from the case $am=aw=0$.
   Not surprisingly, for $am=0.005$, $aw=0.015$ and the critical cases $\kappa=\pm \frac12$, where convergence of the numerical method seems to be lost (see Section~\ref{fivepointfour}),
   the analytical bounds are sharper than the numerically computed bounds. 
\end{rem}

From the data reported in Table~\ref{tab:SFC}, we can implement \eqref{lulo} and compute sharper intervals of enclosure for $\lambda_1$ and $\lambda_2$. Note that we always need information on adjacent eigenvalues:
an upper bound for the one below and a lower bound for the one above.
In order for the enclosures on the right side of \eqref{lulo} to be certified, we also need to ensure that the condition on the left hand side there holds true.
For the data reported in Table~\ref{tab:SFC2}, this is always the case. 

In Table~\ref{tab:SFC2} we show improved inclusions, computed both from the analytical bound and from the numerical bound.
Some of these improved inclusions do not differ significantly, even when the quality of one of the \emph{a priori} bounds  from Table~\ref{tab:SFC} appears to be far lower than the other.
See for example the rows corresponding to $|\kappa|\geq \frac32$.
In these cases the factor $|\Im(z)|^{-2}$ turns out to be far smaller than the coefficient corresponding to the distance to the adjacent points in the spectrum.  By contrast, for the case $|\kappa|=\frac12$, a sharp \emph{a priori} localisation of the adjacent eigenvalues (such as  $\kappa=-\frac12$, $am=0.25$ and $a\omega=0.75$) is critical, because $|\Im(z)|$ is not very small.

\subsection{Global behaviour of the eigenvalues in $aw$ and $am$}

Figure~\ref{clementina} shows $\lambda_{\pm 1}(\frac32, am,aw)$ for a square mesh of 100 equally spaced 
$(a\omega,am)\in [-1,1]\times[0,2]$. The surfaces depicted correspond to an average of the upper and lower bounds for $\lambda_{\pm 1}$ computed directly from \eqref{guayaba}, fixing $h=0.1$. They show the local behaviour of the eigenvalues as functions of $am$ and $a\omega$. On top of the surfaces we also depict the curve (in red) corresponding to the known analytical values for $am=\pm a\omega$ from \eqref{eq:lambdatau}.

\subsection{Optimality of the exponent in Theorem~\ref{mamoncillo}} \label{fivepointfour}
We now test optimality of the leading order of convergence $p(\kappa)$ given in Theorem~\ref{mamoncillo}.
See Figure~\ref{lima}.

For this purpose we have computed residuals of the form
\[
         r_{\kappa}(h)= \left|\lambda_1\left(\kappa,\frac14,\frac14\right)-\tilde{\lambda}_h\right|
\] 
for $h\in\{10^{-3},10^{-2.8},\ldots,10^{-2}\}$. Here $\tilde{\lambda}_h$ is the nearest point (conjugate pair) in
$\spec_{2}(\mAk,\mathcal{L}_h)$ to $\lambda_1(\kappa,\frac14,\frac14)$. According to \eqref{eq:lambdatau},
\[
    \lambda_1\left(\kappa,\frac14,\frac14\right)=\frac12 + \sqrt{ |\kappa|^2 + 2\kappa \frac14 + \frac{1}{16} }.
\]
We have then approximated slopes of the lines
\[
     l_{\kappa}(\log(h))=\log(r_{\kappa}(h)).
\]

In Figure~\ref{lima} we have depicted these slopes,
for 49 equally spaced $\kappa\in(\frac12,3)$. Various conclusions about Theorem~\ref{mamoncillo} can be derived from this figure. 
Taking into account Remark~\ref{parchita}  it appears that an optimal version of \eqref{cemeruco} for $|\kappa|\not=1$ is
\[
      |\lambda_h-\lambda|=\mathcal{O}(h^{\min\{1,|\kappa|-\frac12\}}), \qquad h\to 0.
\]
This is in agreement to the conjecture that the term $\epsilon^{1/2}$ in \eqref{cambur} can be improved to $\epsilon^{1}$.
In such a case, the above conjectured exponent appears to be optimal, in the range $|\kappa|\not \in [1,3/2]$. 

\appendix

\section{Proof of Lemma \ref{lem:decay-eigenfunctions}}
\label{app:frobenius}
According to \cite[Theorem 4.1 in Chap. 4]{CL}, the following holds true.

\begin{thm}
   \label{acai}
   Let $z_0\in\C$. Let $V$ be a complex analytic matrix valued function
   in a neighbourhood of $z_0$.
   If $W$ is a constant $2\times 2$ matrix and the eigenvalues of $W$, $\mu$ and $\nu$, are such that $|\mu-\nu|\notin\N$, then the differential equation
   \begin{align*}
      \left( \frac{\rd}{\rd z} + (z-z_0)^{-1} W + V \right) u = 0
   \end{align*}
   has a fundamental system of the form
   \begin{align*}
      U(z) = (z-z_0)^{-W} P(z)
      % CL, Ch. 4, (1.2)
   \end{align*}
   where $P$ is complex analytic in a neighbourhood of $z_0$, $P(z_0)=\id=\begin{pmatrix} 1 &0\\0&1\end{pmatrix}$ and 
      $(z-z_0)^{-W} := \e^{-\ln(z-z_0)W}$ as in \cite[(1.2) in Chap. 4]{CL}.
\end{thm}

\begin{proof}[Proof of Lemma \ref{lem:decay-eigenfunctions}]
   Without loss of generality we may assume that $\kappa\ge 1/2$. The proof of the case $\kappa\le -1/2$ is analogous.

   Firstly suppose that $2\kappa \not\in \N$. 
   Let $\lambda$ be an eigenvalue of $\mAk$ and let $U$ be a fundamental system of 
   \begin{equation}  \label{gooseberry}   
      (\mfAk-\lambda)U = 0.
   \end{equation}   
   Multiplying \eqref{gooseberry} on the left by 
   $\left(\begin{smallmatrix} 0 & -1 \\ 1 & 0 \end{smallmatrix}\right)$ gives
   \begin{equation}
      \label{araza}
      \left[ 
      \frac{\rd}{\rd\theta} 
      + \left( \frac{1}{\theta} + \frac{1}{\pi-\theta} \right)
      \begin{pmatrix} -\kappa & 0 \\ 0 & \kappa \end{pmatrix}
      + \begin{pmatrix} 0 & -1 \\ 1 & 0 \end{pmatrix}(\pot -\lambda)
	 \right] U
	 = 0,
	 \qquad \theta\in(0,\pi).
   \end{equation}
   By Theorem~\ref{acai}, \eqref{araza} has fundamental systems
   \begin{align}
      \begin{aligned}
      U_0(\theta)
      &=
      \begin{pmatrix}
	 \theta^{-\kappa} & 0 \\ 0 & \theta^{\kappa} 
      \end{pmatrix}
      P_0(\theta),
      \qquad
      U_\pi(\theta)
      =
      \begin{pmatrix}
	 (\pi-\theta)^{-\kappa} & 0 \\ 0 & (\pi-\theta)^{\kappa} 
      \end{pmatrix}
      P_{\pi}(\theta)
      \end{aligned}
   \end{align}
   where $P_0$ is analytic in a neighbourhood of $[0,\pi)$, $P_{\pi}$ is analytic in a neighbourhood of $(0,\pi]$ and $P_0(0) = P_{\pi}(\pi) = \id$.

   Let $u$ be an eigenfunction.
   As $u\in\Ltwotwo$, it follows that there are constants $c_{0}$ and $c_{\pi}$ such that
   \begin{equation*}
      u(\theta) = 
      \begin{pmatrix} \theta^{-\kappa} & 0 \\ 0 & \theta^{\kappa} 
      \end{pmatrix}
      P_0(\theta)
      \begin{pmatrix} 0 \\ c_0
      \end{pmatrix}
      =
      \begin{pmatrix} (\pi-\theta)^{-\kappa} & 0 \\ 0 & (\pi-\theta)^{\kappa} 
      \end{pmatrix}
      P_{\pi}(\theta)
      \begin{pmatrix} 0 \\ c_{\pi}
      \end{pmatrix},
      \quad \theta\in(0,\pi).
   \end{equation*}
   This gives \eqref{uva} under the assumption that $2\kappa \not\in \N$.

   Now assume that $2\kappa \in\N$. We follow a recursive argument.
Set
   \begin{equation*}
   W_0(\theta)
   = \begin{pmatrix} a_0(\theta) & b_0(\theta) \\ c_0(\theta) & d_0(\theta) \end{pmatrix}
   = \frac{1}{\pi-\theta} \begin{pmatrix} -\kappa & 0 \\ 0 & \kappa 
   \end{pmatrix}
   + \begin{pmatrix} 0 & -1 \\ 1 & 0 \end{pmatrix}(\pot(\theta)-\lambda).
   \end{equation*}
Then $a_0, b_0, c_0, d_0$ are analytic functions in $[0,\pi)$.
   Let
   \begin{equation}
      S(\theta) = \begin{pmatrix} \theta & 0 \\ 0 & 1 \end{pmatrix}
      \quad \text{ so that } \quad
      S^{-1}(\theta)S'(\theta) = \begin{pmatrix} \frac{1}{\theta} & 0 \\ 0 & 0 \end{pmatrix},
     \qquad\quad
     \theta\in(0,\,\pi).
   \end{equation}
   The equation \eqref{araza} can be transformed into
   \begin{align}
      \nonumber
      0 &=
      S^{-1}\left[
      \frac{\rd}{\rd\theta} 
      + \frac{1}{\theta}
      \begin{pmatrix} -\kappa & 0 \\ 0 & \kappa
      \end{pmatrix}
      + W_0
      \right]SS^{-1}U
      \\
      \nonumber
      &=
      \left[
      \frac{\rd}{\rd\theta} 
      + S^{-1}S'
      + \frac{1}{\theta}
      \begin{pmatrix} -\kappa & 0 \\ 0 & \kappa
      \end{pmatrix}
      + S^{-1}W_0S
      \right]S^{-1}U
      \\
      \nonumber
      &=
      \left[
      \frac{\rd}{\rd\theta} 
      + \frac{1}{\theta}
      \begin{pmatrix} -\kappa+1 & 0 \\ 0 & \kappa
      \end{pmatrix}
      + 
      \begin{pmatrix} a_0 & \theta^{-1}b_0 \\ \theta c_0 & d_0
      \end{pmatrix}
      \right]S^{-1}U
      \\
      \label{zapote}
      &=
      \left[
      \frac{\rd}{\rd\theta} 
      + \frac{1}{\theta}
      \smash{\underbrace{\begin{pmatrix} -\kappa+1 & b_0(0) \\ 0 & \kappa
      \end{pmatrix}}_{W_1}}
      + 
      \begin{pmatrix} a_0 & \beta_0 \\ \theta c_0 & d_0
      \end{pmatrix}
      \right]S^{-1}U
      \vphantom{
      \underbrace{\begin{pmatrix} -\kappa+1 & b_0(0) \\ 0 & \kappa
      \end{pmatrix}}_{W_1}
      }
   \end{align}
   where $\beta_0(\theta) = \theta^{-1}(b_0(\theta)-b_0(0))$ is analytic in a neighbourhood of $[0,\pi)$.
   In order to diagonalise $W_1$, let
      $T_1 = \begin{pmatrix}
      1 & b_0(0) \\ 0 & 2\kappa-1
   \end{pmatrix}$.
   A further transformation of~\eqref{zapote} gives
   \begin{align*}
      0 &=
      T_1^{-1}\left[
      \frac{\rd}{\rd\theta} 
      + \frac{1}{\theta}
      \begin{pmatrix} -\kappa+1 & b_0(0) \\ 0 & \kappa
      \end{pmatrix}
      + 
      \begin{pmatrix} a_0 & \beta_0 \\ \theta c_0 & d_0
      \end{pmatrix}
      \right]T_1T_1^{-1}S^{-1}U
      \\
      &=
      \left[
      \frac{\rd}{\rd\theta} 
      + \frac{1}{\theta}
      \begin{pmatrix} -\kappa+1 & 0 \\ 0 & \kappa
      \end{pmatrix}
      + 
      \begin{pmatrix} a_1 & b_1 \\ c_1 & d_1
      \end{pmatrix}
      \right]T_1^{-1}S^{-1}U
   \end{align*}
   where $a_1, b_1, c_1, d_1$ are analytic.
   By repeating this process $2\kappa-1$ times we get
   \begin{equation}
      \label{ano}
      0 =
      \left[
      \frac{\rd}{\rd\theta} 
      + \frac{1}{\theta}
      \begin{pmatrix} \kappa-1 & 0 \\ 0 & \kappa
      \end{pmatrix}
      + 
      \begin{pmatrix} a_{2\kappa-1} & b_{2\kappa-1} \\ c_{2\kappa-1} & d_{2\kappa-1}
      \end{pmatrix}
      \right]
      T_{2\kappa-1}^{-1}S^{-1}
      \dots
      T_1^{-1}S^{-1}
      U
   \end{equation}
   where $a_{2\kappa-1}, b_{2\kappa-1}, c_{2\kappa-1}, d_{2\kappa-1}$ are analytic in $[0,\pi)$ and
   \begin{align*}
      T_j = \begin{pmatrix} 1 & b_{j-1}(0) \\ 0 & 2\kappa - j
      \end{pmatrix}.
   \end{align*}
   A final transformation of \eqref{ano} with $S$ yields
   \begin{equation*}
      0 =
      \left[
      \frac{\rd}{\rd\theta} 
      + \frac{1}{\theta}
      \smash{\underbrace{
      \begin{pmatrix} \kappa & b_{2\kappa-1}(0) \\ 0 & \kappa
      \end{pmatrix}}_{ W_{2\kappa}} }
      + 
      \begin{pmatrix} a_{2\kappa-1} & \beta_{2\kappa-1} \\ \theta c_{2\kappa-1} & d_{2\kappa-1}
      \end{pmatrix}
      \right]
      S^{-1}T_{2\kappa-1}^{-1}S^{-1}
      \dots
      T_1^{-1}S^{-1}
      U.
      \vphantom{\underbrace{
      \begin{pmatrix} \kappa & b_{2\kappa-1}(0) \\ 0 & \kappa
      \end{pmatrix}}_{ W_{2\kappa}} }
   \end{equation*}
   
   The eigenvalues of $W_{2\kappa}$ do not differ by a positive integer, therefore the differential equation
      $\displaystyle\left[
      \frac{\rd}{\rd\theta} 
      + \frac{1}{\theta}
      \begin{pmatrix} \kappa & b_{2\kappa-1}(0) \\ 0 & \kappa
      \end{pmatrix}
      + 
      \begin{pmatrix} a_{2\kappa-1} & \beta_{2\kappa-1} \\ \theta c_{2\kappa-1} & d_{2\kappa-1}
      \end{pmatrix}
      \right] Y = 0$
      has a fundamental system of the form 
      $Y(\theta) = \theta^{-\kappa} P_0(\theta)$ for $P_0$ analytic in $[0,\pi)$ and $P_0(0)=\id$.
      Hence a fundamental system of \eqref{araza} is given by
      \begin{align}
	 \nonumber
	 U_0(\theta)
	 &=  S T_1 S T_2 \dots S T_{2\kappa-1} S Y
	 \\
	 \nonumber
	 &=  
	 \theta^{-\kappa}
	 \begin{pmatrix} \theta & 0 \\ 0 & 1
	 \end{pmatrix}
	 \begin{pmatrix} 1 & b_{0}(0) \\ 0 & 2\kappa -1
	 \end{pmatrix}
	 \begin{pmatrix} \theta & 0 \\ 0 & 1
	 \end{pmatrix}
	 \cdots
	 \begin{pmatrix} 1 & b_{2\kappa-2}(0) \\ 0 & 1
	 \end{pmatrix}
	 \begin{pmatrix} \theta & 0 \\ 0 & 1
	 \end{pmatrix}
	 P_0(\theta)
	 \\
	 \label{nispero0}
	 &=
	 \theta^{-\kappa}
	 \begin{pmatrix} \theta^{2\kappa} & p_0(\theta) \\ 0 & (2\kappa-1)!
	 \end{pmatrix}
	 P_0(\theta)
	 =
	 \begin{pmatrix} \theta^{\kappa} & \theta^{-\kappa}p_0(\theta) \\ 0 & \theta^{-\kappa}(2\kappa-1)!
	 \end{pmatrix}
	 P_0(\theta),
      \end{align}
      where $p_0$ is a polynomial of degree $\le 2\kappa -1$.

Now, we can repeat a similar argument at $\theta=\pi$ instead, and find another fundamental system of \eqref{araza} for the segment $(0,\pi]$ of the form
\[
  U_{\pi}(\theta)=
	 \begin{pmatrix} (\pi-\theta)^{-\kappa}(2\kappa-1)! & 0 \\ (\pi-\theta)^{-\kappa}p_{\pi}(\theta) & (\pi-\theta)^{\kappa}
	 \end{pmatrix}
	  P_{\pi}(\theta),
  \]
 where $p_{\pi}$ is a suitable polynomial in $(\pi-\theta)$ of degree $\le 2\kappa-1$ and $P_{\pi}$  is analytic in a neighbourhood of $(0,\pi]$.
      If $u$ is an eigenfunction of $\mAk$, then there are
      constants $c_1, c_2,  d_1$ and $d_2$ such that
      \begin{equation*}
	 u = U_0 \begin{pmatrix} c_1 \\ c_2
	 \end{pmatrix}
	 = U_{\pi} \begin{pmatrix} d_1 \\ d_2
	 \end{pmatrix}.
      \end{equation*}
      By \eqref{nispero0}, and the analogous equation at $\pi$, it follows that, for $u$ to be square integrable, it is necessary that $c_2=d_1=0$.
\end{proof}

\begin{rem} \label{grapefruit}
   The proof shows that all eigenvalues of $\mAk$ are simple.
   Any other solution of $(\mAk-\lambda)u=0$ would diverge of order $-|\kappa|$ for $\theta\to 0$ or $\theta\to\pi$.
\end{rem}
%%%}}}

\section{Computer code}   \label{mandarina} %%%{{{
Complete Comsol LiveLink v4.3b code for computing $\spec_2(\mAk,\mathcal{L}_h)$. See \cite{Comsol}.

\begin{verbatim}
% BASIC_KND_EIGS    Computes conjugate pairs in the second order spectra 
%                   of the angular Kerr-Newman Dirac operator 
%                   for trial spaces made of continuous affine functions
%
% BASIC_KND_EIGS(AM,AW,KAPPA,H,NEVP,SH,RTL)
%    AM    = mass term
%    AW    = energy term
%    KAPPA = angular momentum around axis of symmetry
%    H     = element size
%    NEVP  = number of conjugate pairs
%    SH    = shift
%    RTL   = relative tolerance
%
% Example:
%      z=basic_KND_eigs(0.25,0.75,2.5,0.1,8,0,1E-12)

function z=basic_KND_eigs(am,aw,kappa,h,nevp,sh,rtl)

import com.comsol.model.*
import com.comsol.model.util.*

model = ModelUtil.create('Model');
geom1=model.geom.create('geom1', 1);
mesh1=model.mesh.create('mesh1', 'geom1');

i1=geom1.feature.create('i1', 'Interval');
i1.set('intervals', 'one');
i1.set('p1', '0');
i1.set('p2', 'pi');
geom1.run;

mesh1.automatic(false);
mesh1.feature('size').set('custom', 'on');
mesh1.feature('size').set('hmax', num2str(h));
mesh1.run;

model.param.set('am',num2str(am));
model.param.set('aw',num2str(aw));
model.param.set('kappa',num2str(kappa));
model.param.set('C','am*cos(x)');
model.param.set('S','(kappa/sin(x)+aw*sin(x))');

w=model.physics.create('w', 'WeakFormPDE', 'geom1', {'u1' 'u2'});
w.prop('ShapeProperty').set('shapeFunctionType', 'shlag');
w.prop('ShapeProperty').set('order', 1);
w.feature('wfeq1').set('weak', 1,...
   '(-C*u1+u2x+S*u2)*test(-C*u1+u2x+S*u2)-2*u1t*test(-C*u1+u2x+S*u2)+u1t*test(u1t)');
w.feature('wfeq1').set('weak', 2,...
   '(-u1x+S*u1+C*u2)*test(-u1x+S*u1+C*u2)-2*u2t*test(-u1x+S*u1+C*u2)+u2t*test(u2t)');

cons1=w.feature.create('cons1', 'Constraint',0);
cons1.selection.set([1 2]);
cons1.set('R',1, 'u1^2');
cons1.set('R',2, 'u2^2');

std1=model.study.create('std1');
std1.feature.create('eigv', 'Eigenvalue');
std1.feature('eigv').activate('w', true);

sol1=model.sol.create('sol1');
sol1.study('std1');
sol1.feature.create('st1', 'StudyStep');
sol1.feature('st1').set('study', 'std1');
sol1.feature('st1').set('studystep', 'eigv');
sol1.feature.create('v1', 'Variables');
sol1.feature.create('e1', 'Eigenvalue');
sol1.feature('e1').set('control', 'eigv');
sol1.feature('e1').set('shift', num2str(sh));
sol1.feature('e1').set('neigs', nevp);
sol1.feature('e1').set('rtol', rtl);
sol1.attach('std1');
sol1.runAll;

info= mphsolinfo(model,'soltag','sol1');
z=info.solvals;
\end{verbatim}
%%%}}}

\section*{Acknowledgements}
The authors wish to express their gratitude to P.~Chig\"uiro 
for insightful comments during the preparation of this manuscript.
This research was initiated as part of the programme Spectral Theory of Relativistic Operators
held at the Isaac Newton Institute for Mathematical Sciences in July 2012. 
Subsequent funding was provided by the British Engineering and Physical Sciences Research 
Council  grant EP/I00761X/1, 
FAPA research grant No. PI160322022 of the Facultad de Ciencias de la Universidad de los Andes,
the Carnegie Trust for the Universities of Scotland 
reference 31741, the London Mathematical Society reference 41347 and 
the Edinburgh Mathematical Society Research Support Fund.

\bibliographystyle{alpha}
\bibliography{../Kerr-Newman-estimates}

\clearpage

{\footnotesize
\noindent{\scshape Lyonell Boulton\\
Department of Mathematics \& Maxwell Institute for the Mathematical Sciences\\
Heriot-Watt University\\
Edinburgh, EH14 4AS, United Kingdom}\\
{\itshape E-mail address}: \href{l.boulton@hw.ac.uk}{l.boulton@hw.ac.uk}\\
{\itshape URL}: \url{http://www.ma.hw.ac.uk/~lyonell/}

\bigskip
\bigskip

\noindent{\scshape Monika Winklmeier\newline
Departamento de Matem{\'a}ticas\\
Universidad de Los Andes\\
Cra. 1a No 18A-70, Bogot{\'a}, Colombia}\\
{\itshape E-mail address}: \href{mailto:mwinklme@uniandes.edu.co}{mwinklme@uniandes.edu.co}\\
{\itshape URL}: \url{http://matematicas.uniandes.edu.co/~mwinklme/}
}
\clearpage

%%%{{{ Tables -- Comparison with Chakrabarti
%\newcommand{\cc}[1]{\raisebox{.15ex}{\tiny$#1$}}
\newcommand{\cc}[1]{$#1$}

\definecolor{rlightblue}{rgb}{1.0, 0.90, 0.81}
\definecolor{rcolumbiablue}{rgb}{1.0, 0.87, 0.61}
\definecolor{rbabyblueeyes}{rgb}{.95,.79,.63}
\definecolor{rdarkpastelblue}{rgb}{0.8, 0.62, 0.47}

\definecolor{lightblue}{rgb}{0.81, 0.90, 1.0}
\definecolor{columbiablue}{rgb}{0.61, 0.87, 1.0}
\definecolor{babyblueeyes}{rgb}{.63, .79, .95}
\definecolor{darkpastelblue}{rgb}{0.47, 0.62, 0.8}

\newcommand{\tl}{\cellcolor{rlightblue}}       % too low, 4 decimals agree with our bounds
\newcommand{\ttl}{\cellcolor{rcolumbiablue}}          % too low, 3 decimals agree with our bounds
\newcommand{\tttl}{\cellcolor{rbabyblueeyes}} % too low, 2 decimals agree with our bounds
\newcommand{\ttttl}{\cellcolor{rdarkpastelblue}}            % too low, 1 decimal agree with our bounds

\newcommand{\ti}{\cellcolor{lightblue}}         % too high, 4 decimals agree with our bounds
\newcommand{\tii}{\cellcolor{columbiablue}}     % too high, 3 decimals agree with our bounds
\newcommand{\tiii}{\cellcolor{babyblueeyes}}    % too high, 2 decimals agree with our bounds
\newcommand{\tiiii}{\cellcolor{darkpastelblue}} % too high, 1 decimal agree with our bounds

\newcommand{\ra}[1]{\multirow{2}{*}{$m/\omega =#1$} }

\begin{table}[H]
   \footnotesize
   \renewcommand{\ra}[1]{\multirow{2}{*}{$#1$} }
   \begin{tabular}{l||l|l|l|l|l|l}
      $a\omega$ & $m/\omega = 0$ & $m/\omega = 0.2$ & $m/\omega = 0.4$ & $m/\omega = 0.6$ & $m/\omega = 0.8$ & $m/\omega = 1.0$ \\[1ex] \hline\hline
      \ra{0.1}\rule{0cm}{4ex}
      &    $2.080309          $ &    $2.076445           $ &    $2.072607          $ &    $2.068795          $ &    $2.065008          $ &    $2.061246          $\\
      &    $2.080^{500}_{123} $ &    $2.076^{638}_{260}  $ &    $2.072^{800}_{423} $ &    $2.068^{988}_{610} $ &    $2.06^{5200}_{4823}$ &    $2.061^{438}_{061} $\\[1ex]
      \ra{0.2}
      &    $2.161189          $ &    $2.153720           $ &    $2.146351          $ &    $2.139083          $ &    $2.131917          $ &    $2.124853          $\\
      &    $2.161^{402}_{027} $ &    $2.153^{939}_{563}  $ &    $2.146^{573}_{199} $ &    $2.13^{9305}_{8933}$ &    $2.13^{2137}_{1764}$ &    $2.12^{5067}_{4694}$\\[1ex]
      \ra{0.3}
      &    $2.242573          $ &    $2.231734           $ &    $2.221119          $ &    $2.210730          $ &    $2.200569          $ &    $2.190635          $\\
      &    $2.242^{851}_{476} $ &    $2.23^{2029}_{1655} $ &    $2.221^{425}_{049} $ &    $2.21^{1035}_{0663}$ &    $2.200^{863}_{494} $ &    $2.190^{910}_{540} $\\[1ex]
      \ra{0.4}
      &\ttl$2.324\cc{395}     $ & \tl$2.3104\cc{02}      $ & \tl$2.2968\cc{06}     $ &\ttl$2.283\cc{610}     $ & \tl$2.2708\cc{15}     $ & \tl$2.2584\cc{19}     $\\
      &    $2.324^{801}_{427} $ &    $2.310^{849}_{472}  $ &    $2.29^{7271}_{6899}$ &    $2.28^{4073}_{3702}$ &    $2.27^{1254}_{0880}$ &    $2.258^{810}_{436} $\\[1ex]
      \ra{0.5}
      &\ttl$2.406\cc{589}     $ &\ttl$2.389\cc{642}      $ &\ttl$2.373\cc{312}     $ &\ttl$2.357\cc{605}     $ &\ttl$2.342\cc{520}     $ &\ttl$2.328\cc{049}     $\\
      &    $2.40^{7214}_{6840}$ &    $2.3^{90339}_{89965}$ &    $2.37^{4040}_{3672}$ &    $2.35^{8325}_{7953}$ &    $2.34^{3184}_{2810}$ &    $2.328^{615}_{239} $\\[1ex]
      \ra{0.6}
      &  \ttl$2.489\cc{091}      $ &\ttttl$2.4\cc{69373}      $ &\tttl$2.45\cc{0543}      $ &\tttl$2.43\cc{2607}      $ & \tttl$2.41\cc{5559}      $ &  \ttl$2.399\cc{378}       $\\
      &      $2.4^{90049}_{89676}$ &      $2.470^{448}_{076}  $ &     $2.451^{665}_{295}  $ &     $2.433^{700}_{328}  $ &      $2.416^{545}_{169}  $ &      $2.^{400186}_{399814}$\\[1ex]
      \ra{0.7}
      & \tttl$2.57\cc{1837}      $ &\ttttl$2.5\cc{49516}      $ & \ttl$2.52\cc{8407}      $ & \ttl$2.50\cc{8514}      $ &\ttttl$2.4\cc{89817}      $ & \tttl$2.47\cc{2274}       $\\
      &      $2.57^{3273}_{2897} $ &      $2.55^{1127}_{0758} $ &     $2.5^{30079}_{29708}$ &     $2.5^{10118}_{09745}$ &      $2.49^{1231}_{0858} $ &      $2.473^{400}_{027}   $\\[1ex]
      \ra{0.8}
      & \tttl$2.65\cc{4763}      $ &\ttttl$2.6\cc{29996}      $ &\tttl$2.60\cc{6820}      $ &\tttl$2.58\cc{5231}      $ & \tttl$2.56\cc{5189}      $ & \tttl$2.54\cc{6616}       $\\
      &      $2.656^{846}_{473}  $ &      $2.63^{2332}_{1961} $ &     $2.60^{9221}_{8852} $ &     $2.587^{503}_{129}  $ &      $2.56^{7151}_{6778} $ &      $2.54^{8137}_{7763}  $\\[1ex]
      \ra{0.9}
      &\ttttl$2.7\cc{37803}      $ & \tttl$2.71\cc{0737}      $ &\tttl$2.68\cc{5697}      $ &\tttl$2.66\cc{2669}      $ & \tttl$2.64\cc{1580}      $ & \tttl$2.62\cc{2294}       $\\
      &      $2.740^{741}_{368} $ &      $2.71^{4015}_{3646} $ &     $2.68^{9038}_{8666} $ &     $2.665^{783}_{411}  $ &      $2.64^{4218}_{3841} $ &      $2.62^{4288}_{3912}  $\\[1ex]
      \ra{1.0}
      & \tttl$2.82\cc{0892}      $ &  \tttl$2.79\cc{1662}     $ & \tttl$2.76\cc{4958}     $ & \tttl$2.74\cc{0745}     $ & \ttttl$2.7\cc{18899}     $ &\ttttl$2.\cc{699206}       $\\
      &      $2.824^{924}_{551}  $ &       $2.79^{6140}_{5767}$ &      $2.769^{476}_{101} $ &      $2.744^{894}_{523} $ &       $2.72^{2345}_{1971}$ &      $2.701^{750}_{374}   $
   \end{tabular}
   \bigskip
   \bigskip
\caption{\label{table:chakra3by2p}
Computation of $|\lambda_{-1}(3/2,am,a\omega)|$ for different $a\omega$ and $\omega/m$, as shown.
The quantities in the upper part of each row are the positive square root of those in
\protect\cite[Table 2b]{chakrabarti}.
The quantities in the lower part of each row are the enclosures determined directly from an application of \eqref{guayaba}.
Quantities on the upper rows which are not within our guaranteed error bounds are shaded.
}

\end{table}

\pagebreak

\begin{table}[H]
\footnotesize
% kappa = -3/2, Chakarbarti, upper and lower bounds
\renewcommand{\ra}[1]{\multirow{2}{*}{$#1$} }
\begin{tabular}{l||l|l|l|l|l|l}
   $a\omega$ & $m/\omega = 0$ & $m/\omega = 0.2$ & $m/\omega = 0.4$ & $m/\omega = 0.6$ & $m/\omega = 0.8$ & $m/\omega = 1.0$ \\[1ex] \hline\hline
   \ra{0.1}\rule{0cm}{4ex}
   &    $1.920331         $  &    $1.924477           $ &    $1.928648           $ &    $1.932845          $ &    $1.937067          $ &     $1.941315         $\\
   &    $1.920^{516}_{141}$  &    $1.924^{659}_{287}  $ &    $1.928^{830}_{457}  $ &    $1.93^{3027}_{2653}$ &    $1.93^{7249}_{6876}$ &     $1.941^{497}_{125}$\\[1ex]
   \ra{0.2}
   &    $1.841373         $  &    $1.849972           $ &    $1.858676           $ &    $1.867484          $ &    $1.876395          $ &     $1.885406         $\\
   &    $1.841^{536}_{163}$  &    $1.8^{50129}_{49755}$ &    $1.858^{828}_{454}  $ &    $1.867^{632}_{259} $ &    $1.876^{543}_{170} $ &     $1.885^{559}_{185}$\\[1ex]
   \ra{0.3}
   &    $1.763193          $ &    $1.776584           $ &    $1.790217           $ &    $1.804084          $ &    $1.818181          $ &     $1.832498         $\\
   &    $1.76^{3306}_{2931}$ &    $1.776^{673}_{299}  $ &    $1.7^{90285}_{89910}$ &    $1.80^{4140}_{3765}$ &    $1.81^{8236}_{7859}$ &     $1.832^{569}_{192}$\\[1ex]
   \ra{0.4}
   &    $1.685863          $ &\tii$1.704\cc{417}      $ &\tii$1.723\cc{408}      $ &\tii$1.742\cc{821}     $ &\tii$1.762\cc{635}     $ & \tii$1.782\cc{831}    $\\
   &    $1.685^{883}_{507} $ &    $1.704^{377}_{002}  $ &    $1.72^{3316}_{2939} $ &    $1.742^{693}_{314} $ &    $1.762^{500}_{122} $ &     $1.782^{732}_{353}$\\[1ex]
   \ra{0.5}
   &\tii$1.609\cc{449}       $ &\tii$1.633\cc{573}      $ &\tii $1.658\cc{395}     $ &\tii$1.683\cc{877}       $ &  \tii$1.709\cc{979}     $ &  \tii$1.736\cc{651}     $\\
   &    $1.60^{9331}_{8955}  $ &    $1.63^{3333}_{2955} $ &     $1.65^{8041}_{7660}$ &    $1.683^{441}_{060}   $ &      $1.709^{519}_{138} $ &      $1.73^{6258}_{5878}$\\[1ex]
   \ra{0.6}
   & \tiii$1.53\cc{4014}     $ & \tiii$1.56\cc{4156}    $ &\tiii$1.59\cc{5324}     $ & \tiii$1.62\cc{7446}     $ &\tiiii$1.6\cc{60439}     $ & \tiii$1.69\cc{4209}     $\\
   &      $1.533^{718}_{340} $ &      $1.563^{634}_{255}$ &     $1.594^{586}_{204} $ &      $1.626^{545}_{163} $ &      $1.659^{483}_{099} $ &      $1.69^{3364}_{2979}$\\[1ex]
   \ra{0.7}
   &  \tii$1.459\cc{615}     $ & \tiii$1.49\cc{6266}    $ &\tiii$1.53\cc{4344}     $ & \tiii$1.57\cc{3721}     $ & \tiii$1.61\cc{4248}     $ & \tiii$1.65\cc{5756}     $\\
   &      $1.45^{9117}_{8734}$ &      $1.495^{386}_{003}$ &     $1.53^{3088}_{2702}$ &      $1.57^{2172}_{1786}$ &      $1.612^{585}_{196} $ &      $1.65^{4261}_{3871}$\\[1ex]
   \ra{0.8}
   & \tiii$1.38\cc{6295}     $ & \tiii$1.42\cc{9994}    $ &\tiii$1.47\cc{5600}     $ & \tiii$1.52\cc{2898}     $ &\tiiii$1.5\cc{71636}     $ &\tiiii$1.6\cc{21537}     $\\
   &      $1.385^{601}_{219} $ &      $1.428^{695}_{310}$ &     $1.473^{686}_{299} $ &      $1.520^{493}_{104} $ &      $1.56^{9022}_{8628}$ &      $1.61^{9158}_{8766}$\\[1ex]
   \ra{0.9}
   & \tiii$1.31\cc{4077}     $ & \tiii$1.36\cc{5418}    $ &\tiii$1.41\cc{9230}     $ & \tiii$1.47\cc{5166}     $ &\tiiii$1.5\cc{32828}     $ &\tiiii$1.5\cc{91781}     $\\
   &      $1.31^{3256}_{2870}$ &      $1.363^{678}_{289}$ &     $1.416^{530}_{140} $ &      $1.471^{685}_{292} $ &      $1.528^{987}_{591} $ &      $1.58^{8261}_{7862}$\\[1ex]
   \ra{1.0}
   & \tii$1.242\cc{955}     $ & \tiii$1.30\cc{2592}    $ &\tiii$1.36\cc{5353}     $ &\tiiii$1.4\cc{30702}     $ & \tiii$1.49\cc{8032}     $ & \tiii$1.56\cc{6699}     $\\
   &      $1.24^{2165}_{1776}$ &      $1.300^{452}_{060}$ &     $1.361^{773}_{379} $ &      $1.425^{926}_{528} $ &      $1.492^{675}_{274} $ &      $1.561^{754}_{352} $
\end{tabular}
\bigskip
\bigskip

\caption{\label{table:chakra3by2m}
Computation of $|\lambda_{-1}(-3/2,am,a\omega)|$ for different $a\omega$ and $\omega/m$, as shown.
The quantities in the upper part of each row are the positive square root of those in \protect\cite[Table 2b]{chakrabarti}.
The quantities in the lower part of each row are the enclosures determined directly from an application of \eqref{guayaba}.
Quantities on the upper rows which are not within our guaranteed error bounds are shaded.
}
\end{table}
%%%}}}

\pagebreak

%%%{{{ Tables -- Comparison with SFC83
\begin{table}[H]

\renewcommand{\arraystretch}{1.5}

\setlength{\tabcolsep}{9pt}
\hspace*{-1.5cm}
\begin{tabular}{r|ll|ll|ll|ll}
   \hline\hline
   \multicolumn{8}{c}{$am = 0.005,\ a\omega = 0.015$}\\
   \hline\hline
   \multicolumn{8}{c}{}\\
   & \multicolumn{2}{c|}{$n=-1$}& \multicolumn{2}{c|}{$n=1$}& \multicolumn{2}{c|}{$n=2$}& \multicolumn{2}{c}{$n=3$}\\[2ex]
   & \multicolumn{1}{c}{\footnotesize A} & \multicolumn{1}{c|}{\footnotesize N}
   & \multicolumn{1}{c}{\footnotesize A} & \multicolumn{1}{c|}{\footnotesize N}
   & \multicolumn{1}{c}{\footnotesize A} & \multicolumn{1}{c|}{\footnotesize N}
   & \multicolumn{1}{c}{\footnotesize A} & \multicolumn{1}{c}{\footnotesize N}
   \\[1ex]
   \hline
   \hspace*{-.3cm}
   $\kappa= -4.5$ & $-4.9_{9299}^{7997}$    &$-4.98_{817}^{547}$        & $4.9_{7997}^{9299}$    &  $4.98_{456}^{727}$        & $5.9_{8248}^{9500}$    & $5.9_{8497}^{9176}   $   & $6.9_{8427}^{9643}$    & $6.9_{8395}^{9622}$     \\
   $        -3.5$ & $-3.9_{9374}^{7997}$    &$-3.98_{833}^{611}$        & $3.9_{7997}^{9374}$    &  $3.98_{500}^{723}$        & $4.9_{8298}^{9600}$    & $4.9_{8620}^{9190}   $   & $5.9_{8499}^{9750}$    & $5.9_{8570}^{9621}$     \\
   $        -2.5$ & $-2.9_{9499}^{7996}$    &$-2.98_{874}^{698}$        & $2.9_{7996}^{9499}$    &  $2.98_{555}^{731}$        & $3.9_{8373}^{9750}$    & $3.9_{8776}^{9242}   $   & $4.9_{8599}^{9900}$    & $4.9_{8778}^{9659}$     \\
   $        -1.5$ & $-1.9_{9749}^{7994}$    &$-1.98_{990}^{812}$        & $1.9_{7994}^{9749}$    &  $1.98_{612}^{789}$        & $_{2.98498}^{3.00000}$ & $2.9_{8971}^{9404}   $   & $_{3.98749}^{4.00125}$ & $3.9_{9005}^{9808}$     \\
   $        -0.5$ & $-_{1.00500}^{0.97988}$ &$-_{1.17342}^{0.81076}$    & $_{0.97988}^{1.00500}$ &  $_{0.80779}^{1.16967}$    & $_{1.98748}^{2.00500}$ & $_{1.74245}^{2.25132}$   & $_{2.98999}^{3.00500}$ & $_{2.68941}^{3.30912}$  \\
   $         0.5$ & $-_{1.01990}^{0.99500}$ &$-_{1.18847}^{0.82905}$    & $_{0.99500}^{1.01990}$ &  $_{0.83197}^{1.19218}$    & $_{1.99500}^{2.01249}$ & $_{1.75063}^{2.26051}$   & $_{2.99500}^{3.01000}$ & $_{2.69440}^{3.31501}$  \\
   $         1.5$ & $-2.0_{1995}^{0248}$    &$-2.01_{189}^{013}$        & $2.0_{0248}^{1995}$    &  $2.01_{212}^{389}$        & $_{2.99999}^{3.01499}$ & $3.0_{0600}^{1032}   $   & $_{3.99874}^{4.01250}$ & $4.00_{196}^{998}$      \\
   $         2.5$ & $-3.0_{1997}^{0498}$    &$-3.01_{303}^{127}$        & $3.0_{0498}^{1997}$    &  $3.01_{270}^{446}$        & $4.0_{0249}^{1624}$    & $4.0_{0760}^{1226}   $   & $5.0_{0099}^{1400}$    & $5.0_{0343}^{1225}$     \\
   $         3.5$ & $-4.0_{1998}^{0623}$    &$-4.01_{389}^{167}$        & $4.0_{0623}^{1998}$    &  $4.01_{278}^{500}$        & $5.0_{0399}^{1699}$    & $5.0_{0812}^{1382}   $   & $6.0_{0249}^{1500}$    & $6.0_{0381}^{1432}$     \\
   $         4.5$ & $-5.0_{1998}^{0698}$    &$-5.01_{454}^{183}$        & $5.0_{0698}^{1998}$    &  $5.01_{274}^{545}$        & $6.0_{0499}^{1749}$    & $6.0_{0825}^{1504}   $   & $7.0_{0356}^{1571}$    & $7.0_{0379}^{1607}$    
\end{tabular}

\bigskip
\bigskip
\bigskip
\bigskip

\setlength{\tabcolsep}{9pt}
\hspace*{-1.5cm}
\begin{tabular}{r|ll|ll|ll|ll}
   \hline\hline
   \multicolumn{8}{c}{$am = 0.25,\ a\omega = 0.75$}\\
   \hline\hline
   \multicolumn{8}{c}{}\\
   & \multicolumn{2}{c|}{$n=-1$}& \multicolumn{2}{c|}{$n=1$}& \multicolumn{2}{c|}{$n=2$}& \multicolumn{2}{c}{$n=3$}\\[2ex]
   & \multicolumn{1}{c}{\footnotesize A} & \multicolumn{1}{c|}{\footnotesize N}
   & \multicolumn{1}{c}{\footnotesize A} & \multicolumn{1}{c|}{\footnotesize N}
   & \multicolumn{1}{c}{\footnotesize A} & \multicolumn{1}{c|}{\footnotesize N}
   & \multicolumn{1}{c}{\footnotesize A} & \multicolumn{1}{c}{\footnotesize N}
   \\[1ex]
   \hline
   \hspace*{-.3cm}
   $\kappa= -4.5$  & $-_{4.61607}^{3.93330}$ & $-4.3_{5071}^{4800}$    & $_{3.93330}^{4.61607}$ & $4.29_{622}^{891}$  &    $5._{08853}^{73293}$   & $5.4_{2561}^{3238}$      & $6._{19204}^{81221}$   & $6.5_{1143}^{2368}$     \\
   $        -3.5$  & $-_{3.65037}^{2.91227}$ & $-3.37_{482}^{259}$     & $_{2.91227}^{3.65037}$ & $3.30_{759}^{981}$  &    $4._{10889}^{78460}$   & $4.4_{6547}^{7115}$      & $5._{22722}^{86806}$   & $5.5_{6039}^{7087}$     \\
   $        -2.5$  & $-_{2.71222}^{1.87132}$ & $-2.41_{438}^{258}$     & $_{1.87132}^{2.71222}$ & $2.32_{570}^{746}$  &    $3._{14116}^{86421}$   & $3.5_{2756}^{3220}$      & $4._{27769}^{94708}$   & $4.63_{012}^{890}$     \\
   $        -1.5$  & $-_{1.85079}^{0.75000}$ & $-1.4_{9000}^{8796}$    & $_{0.75000}^{1.85079}$ & $1.3_{5885}^{6075}$ &    $_{2.19948}^{3.00000}$ & $2.6_{3770}^{4203}$      & $_{3.35555}^{4.06609}$ & $3.7_{3691}^{4491}$     \\
   $        -0.5$  & $-_{1.28078}^{0.25000}$ & $-0._{90569}^{44078}$   & $_{0.25000}^{1.28078}$ & $0._{22804}^{65248}$&    $_{1.33113}^{2.26557}$ & $_{1.65046}^{2.12229}$   & $_{2.48861}^{3.26040}$ & $_{2.62846}^{3.22283}$  \\
   $         0.5$  & $-_{1.85079}^{0.75000}$ & $-1._{62848}^{32502}$   & $_{0.75000}^{1.85079}$ & $1._{42913}^{76746}$&    $_{1.75000}^{2.60850}$ & $2._{02453}^{55401}$     & $_{2.75000}^{3.50000}$ & $_{2.87169}^{3.51065}$  \\
   $         1.5$  & $-2._{90754}^{09520}$   & $-2.57_{743}^{578}$     & $2._{09520}^{90754}$   & $2.65_{566}^{737}$  &    $_{2.99037}^{3.72312}$ & $3.4_{3822}^{4251}$      & $_{3.93330}^{4.61607}$ & $4.3_{2596}^{3398}$     \\
   $         2.5$  & $-3._{93274}^{21410}$   & $-3.62_{304}^{132}$     & $3._{21410}^{93274}$   & $3.68_{142}^{315}$  &    $4._{10889}^{784591}$  & $4.51_{313}^{774}$       & $5._{04150}^{68715}$   & $5.4_{0876}^{1754}$     \\
   $         3.5$  & $-4._{94708}^{27769}$   & $-4.64_{965}^{746}$     & $4._{27769}^{94708}$   & $4.69_{575}^{794}$  &    $5._{18139}^{82338}$   & $5.5_{5791}^{6358}$      & $6._{11396}^{73557}$   & $6.4_{6318}^{7365}$     \\
   $         4.5$  & $-5._{95636}^{31776}$   & $-5.66_{717}^{448}$     & $5._{31776}^{95636}$   & $5.70_{488}^{757}$  &    $6._{23074}^{85020}$   & $6.5_{8784}^{9460}$      & $7._{16619}^{77081}$   & $7.5_{0180}^{1404}$     \\
\end{tabular}
\bigskip
\bigskip
\caption{
\label{tab:SFC}
Analytic (A) and numeric (N) bounds for the eigenvalues of $\mAk$. The range of $\kappa$, $n$, $am$ and $a\omega$ employed corresponds to analogous calculations in \protect\cite[Table~II]{SFC}. Here the numerical bounds were determined directly from \eqref{guayaba} and their computation does not use any input from the analytical bounds.
}
\end{table}
%%%}}}

\pagebreak

%%%{{{ Tables -- Comparison with SFC83 -- improved numerical bounds
\begin{table}[H]

\renewcommand{\arraystretch}{1.5}

\begin{tabular}{r|lll|lll}
   \hline\hline
   \multicolumn{7}{c}{$am = 0.005,\ a\omega = 0.015$}\\
   \hline\hline
   \multicolumn{7}{c}{}\\
   & \multicolumn{3}{c|}{$n=1$}& \multicolumn{3}{c}{$n=2$}\\[2ex]
   & \parbox{1.2cm}{\footnotesize  from A}
   & \parbox{1.2cm}{\footnotesize  from N} & \cite{SFC}
   & \parbox{1.2cm}{\footnotesize  from A}
   & \parbox{1.2cm}{\footnotesize  from N} & \cite{SFC}
   \\[1ex]
   \hline
   $\kappa= -4.5$ & $4.9859_{0}^{2}$     & $4.9859_{0}^{2}$       & \ttl$4.98581$  & $5.9883_{4}^{8}$        & $5.9883_{4}^{8}$       &           \\
   $        -3.5$ & $3.9861_{1}^{2}$     & $3.9861_{1}^{2}$       & $3.98611$      & $4.9890_{3}^{6}$        & $4.9890_{3}^{6}$       & $4.98905$ \\
   $        -2.5$ & $2.9864_{2}^{4}$     & $2.9864_{2}^{4}$       & $2.98643$      & $3.990_{08}^{10}$       & $3.990_{08}^{10}$      & $3.99009$ \\
   $        -1.5$ & $1.9870_{0}^{1}$     & $1.9870_{0}^{1}$       & $1.98700$      & $2.9918_{6}^{8}$        & $2.9918_{6}^{8}$       & $2.99187$ \\
   $        -0.5$ & $_{0.95595}^{1.00537}$ & $_{0.94529}^{1.00693}$ & $0.98843$      & $_{1.93169}^{2.06215}$  & $_{1.90340}^{2.07515}$ & $1.99567$ \\
   $         0.5$ & $_{0.97907}^{1.02824}$ & $_{0.96816}^{1.02970}$ & $1.01167$      & $_{1.93988}^{2.07151}$  & $_{1.91121}^{2.08548}$ & $2.00435$ \\
   $         1.5$ & $2.0130_{0}^{1}$     & $2.0130_{0}^{1}$       & $2.01300$      & $3.0081_{5}^{7}$        & $3.0081_{5}^{7}$       & $3.00815$ \\
   $         2.5$ & $3.0135_{7}^{8}$     & $3.0135_{7}^{8}$       & $3.01357$      & $4.0099_{2}^{4}$        & $4.0099_{2}^{4}$       & $4.00993$ \\
   $         3.5$ & $4.013_{88}^{90}$     & $4.013_{88}^{90}$      & $4.01389$      & $5.0109_{5}^{8}$        & $5.0109_{5}^{8}$       & $5.01096$ \\
   $         4.5$ & $5.014_{08}^{10}$     & $5.014_{08}^{10}$      & $5.01409$      & $6.0116_{3}^{6}$        & $6.0116_{3}^{6}$       &           
\end{tabular}

\bigskip
\bigskip
\bigskip
\bigskip

\begin{tabular}{r|lll|lll}
   \hline\hline
   \multicolumn{7}{c}{$am = 0.25,\ a\omega = 0.75$}\\
   \hline\hline
   \multicolumn{7}{c}{}\\
   & \multicolumn{3}{c|}{$n=1$}& \multicolumn{3}{c}{$n=2$}\\[2ex]
   & \parbox{1.2cm}{\footnotesize  from A}
   & \parbox{1.2cm}{\footnotesize  from N} & \cite{SFC}
   & \parbox{1.2cm}{\footnotesize  from A}
   & \parbox{1.2cm}{\footnotesize  from N} & \cite{SFC}
   \\[1ex]
   \hline
   $\kappa= -4.5$ &$4.2975_{5}^{7}$     &$4.2975_{6}^{7}$      & $4.29756$ &     $5.42_{898}^{902}$    &$5.42_{898}^{901}$    &          \\
   $        -3.5$ &$3.308_{69}^{70}$    &$3.308_{69}^{70}$     & $3.30870$ &     $4.468_{29}^{32}$     &$4.4683_{0}^{2}$      & \ttl$4.46676$\\
   $        -2.5$ &$2.3265_{7}^{8}$     &$2.3265_{7}^{8}$      & $2.32657$ &     $3.5298_{6}^{9}$      &$3.5298_{7}^{9}$      & \ttl$3.52651$\\
   $        -1.5$ &$1.359_{79}^{80}$    &$1.359_{79}^{80}$     & \tii$1.35984$ & $2.6398_{5}^{7}$      &$2.6398_{5}^{7}$      & \ttl $2.63036$\\
   $        -0.5$ &$0._{38970}^{54256}$ &$0.4_{0304}^{9138}$  & $0.44058$ &     $1.79_{396}^{828}$    &$1._{81137}^{93148}$  & $1.84225$\\
   $         0.5$ &$1._{40966}^{61048}$ &$1._{53115}^{60808}$  & $1.59764$ &     $2._{13715}^{44911}$  &$2._{16893}^{42358}$  & $2.22587$\\
   $         1.5$ &$2.6565_{0}^{1}$     &$2.6565_{0}^{1}$      & \tii$2.65654$ & $3.4403_{5}^{8}$      &$3.4403_{5}^{7}$      & \ttl$3.43391$\\
   $         2.5$ &$3.6822_{8}^{9}$     &$3.6822_{8}^{9}$      & $3.68229$ &     $4.5154_{2}^{5}$      &$4.5154_{2}^{4}$      & \ttl$4.51300$\\
   $         3.5$ &$4.6968_{4}^{5}$     &$4.6968_{4}^{5}$      & $4.69685$ &     $5.5607_{2}^{6}$      &$5.5607_{3}^{5}$      & \ttl$5.55956$\\
   $         4.5$ &$5.7062_{1}^{3}$     &$5.7062_{1}^{3}$      & $5.70622$ &     $6.591_{19}^{24}$     &$6.5912_{0}^{3}$      &          \\
\end{tabular}

\bigskip
\bigskip

\caption{
\label{tab:SFC2}
Improved numerical enclosures for the eigenvalues originally reported in \protect\cite[Table~II]{SFC}.
These improved bounds were found from the data in Table~\ref{tab:SFC}, analytical or numerical as appropriate, and by means of an implementation of  \eqref{lulo}. Values from \protect\cite{SFC} that are over or under-shot are highlighted in colour. }
\end{table}
%%%}}}

\pagebreak

\begin{figure}[H]
\includegraphics[width=10cm]{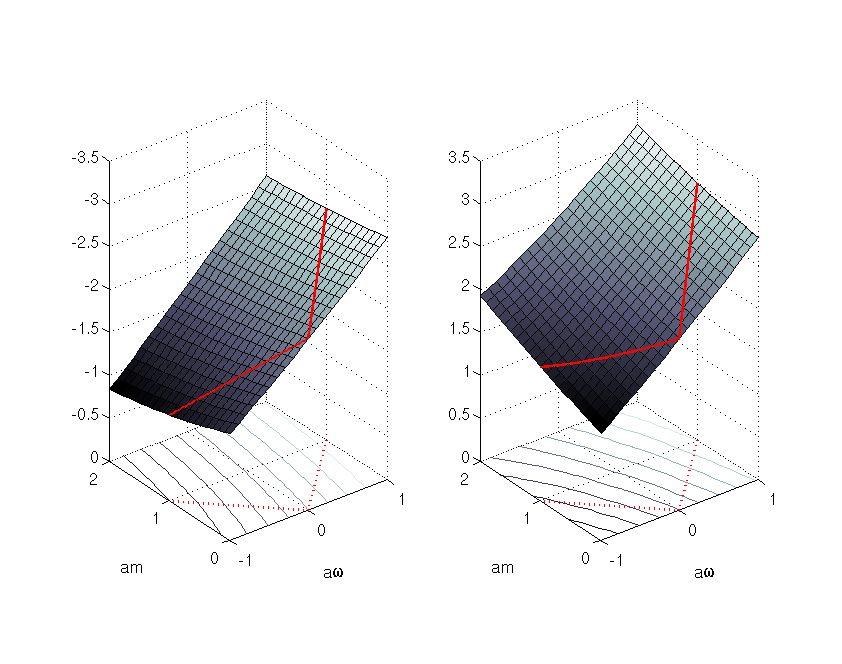}
\caption{Approximation of $\lambda_{\pm1}(\frac32,am,a\omega)$ for 100 different $(a\omega,am)$ equally distributed in the rectangle $[-1,1]\times [0,2]$. 
The red curve corresponds to the exact value of $\lambda_{\pm 1}$ from \eqref{eq:lambdatau}.}\label{clementina}
\end{figure}

\begin{figure}[H]
\includegraphics[width=10cm]{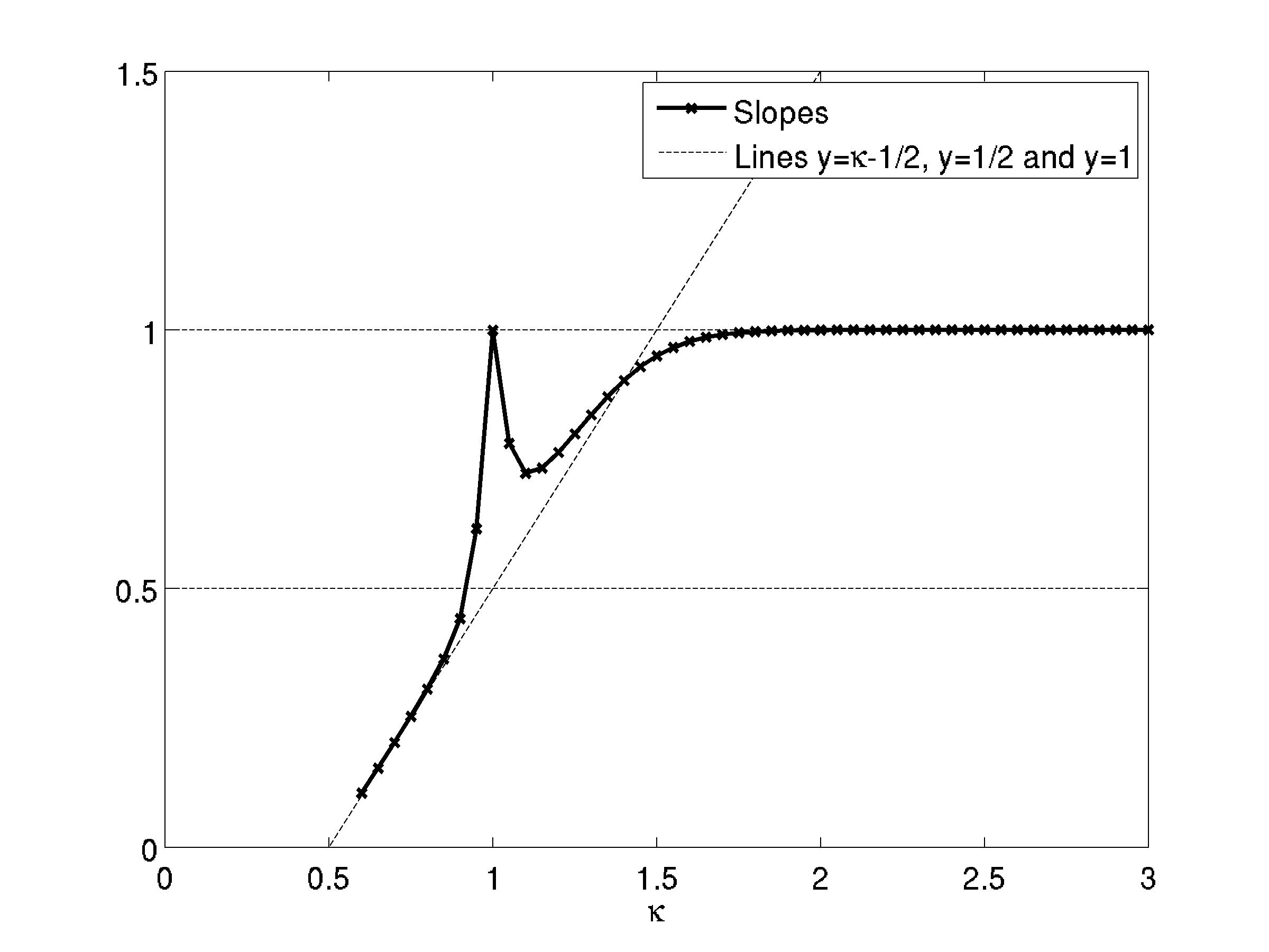}
\caption{A numerical approximation of the optimal exponent in Theorem~\ref{mamoncillo}. } \label{lima}
\end{figure}

\end{document}